\documentclass[12pt]{amsart}

\usepackage{amsmath}
\usepackage{amsfonts}
\usepackage{amssymb}
\usepackage{latexsym}
\usepackage{graphicx}
\usepackage{enumerate}
\usepackage{multirow}
\usepackage[usenames,dvipsnames]{color}

\usepackage{float}
\usepackage{graphicx}
\usepackage{caption}
\usepackage{subcaption}

\newtheorem{theorem}{Theorem} [section]
\newtheorem{proposition}[theorem]{Proposition}
\newtheorem{lemma}[theorem]{Lemma}

\newtheorem{remark}[theorem]{Remark}
\newtheorem{definition}[theorem]{Definition}
\newtheorem{assumption}[theorem]{Hypothesis}

\numberwithin{equation}{section}

\usepackage{mathtools}

\makeatletter
\DeclareRobustCommand\widecheck[1]{{\mathpalette\@widecheck{#1}}}
\def\@widecheck#1#2{%
    \setbox\z@\hbox{\m@th$#1#2$}%
    \setbox\tw@\hbox{\m@th$#1%
       \widehat{%
          \vrule\@width\z@\@height\ht\z@
          \vrule\@height\z@\@width\wd\z@}$}%
    \dp\tw@-\ht\z@
    \@tempdima\ht\z@ \advance\@tempdima2\ht\tw@ \divide\@tempdima\thr@@
    \setbox\tw@\hbox{%
       \raise\@tempdima\hbox{\scalebox{1}[-1]{\lower\@tempdima\box
\tw@}}}%
    {\ooalign{\box\tw@ \cr \box\z@}}}
\makeatother

\newcommand{\ep}{\epsilon}
\newcommand{\om}{\omega}

\renewcommand{\P}{{\mathbb P}}

\DeclareMathOperator{\sech}{sech}
\DeclareMathOperator{\Res}{Res}

\DeclareMathOperator{\sgn}{sgn}

\newcommand{\ds}{\displaystyle}

\newcommand{\be}{\begin{equation}}
\newcommand{\ee}{\end{equation}}
\newcommand{\bes}{\begin{equation*}}
\newcommand{\ees}{\end{equation*}}
\newcommand{\mand}{\quad \text{and}\quad}

\newcommand{\R}{{\bf{R}}}

\newcommand{\Z}{{\bf{Z}}}

\newcommand{\cSp}{{{S}_j^+}}
\newcommand{\cSm}{{{S}_j^-}}

\newcommand{\A}{{\mathcal{A}}}
\newcommand{\B}{{\mathcal{B}}}
\newcommand{\V}{{\mathcal{V}}}
\renewcommand{\O}{{\mathcal{O}}}

\newcommand{\I}{{\mathcal{I}}}

\newcommand{\Ex}{{\mathbb{E}}}

\renewcommand{\H}{{\mathcal{H}}}

\newcommand{\U}{{\mathcal{V}}}
\newcommand{\W}{{\mathcal{W}}}

\newcommand{\Pb}{\bar{P}}
\newcommand{\Qb}{\bar{Q}}
\newcommand{\Fb}{\bar{F}}
\newcommand{\Gb}{\bar{G}}

\renewcommand{\tilde}{\widetilde}

\renewcommand{\check}{\widecheck}

\newcommand{\bunderbrace}[2]{%
  \begin{array}[t]{@{}c@{}}
  \underbrace{#1}\\
  #2
  \end{array}
}

\title{Approximation of (some) Random FPUT Lattices by KdV Equations}

\author{Joshua A.~McGinnis}
 \address{University of Pennsylvania, Philadelphia}
 \author{J.~Douglas Wright}
 \address{Drexel University, Philadelphia}

\begin{document}
\begin{abstract} We consider a Fermi-Pasta-Ulam-Tsingou lattice with randomly varying coefficients. 
We discover a relatively simple condition which when placed on the nature of the randomness allows us to prove that small amplitude/long wavelength solutions are almost surely rigorously approximated by solutions of Korteweg-de Vries equations for very long times. The key ideas combine energy estimates with homogenization theory and the technical proof requires a novel application of autoregressive processes.
\end{abstract}
\maketitle

\section{Introduction}\label{introduction}

Consider a variable mass Fermi-Pasta-Ulam-Tsingou (FPUT) lattice\footnote{We write the equations as a first order system as opposed to its possibly more familiar second order form
\bes
m(j) \ddot{x}(j) = \V'( x(j+1) - x(j))-\V'(x(j) - x(j-1)).
\ees
The change of variables leading from this to \eqref{FPUT} is $q(j) = x(j+1)-x(j)$ and $p(j) = \dot{x}(j)$.}:
\be\label{FPUT}
\dot{q}(j,t) = \delta^+ p (j,t) \mand m(j)\dot{p}(j,t) = \delta^- [ \V'(q) ](j,t).
\ee
Here $t \in \R$, $j \in \Z$ and the unknowns $q(j,t)$ (the {\it relative displacement}) and $p(j,t)$ (the {\it velocity}) are real-valued. The {\it mass coefficients} $m(j)$ are strictly positive
and \be
\label{homog spring}\V(q):={1 \over 2} q^2 + {1 \over 3} q^3\ee
is the {\it spring potential}\footnote{This choice of the spring potential---which is an instance of the ``$\alpha$-potential'' from \cite{fermi-pasta-ulam}---is made mainly for simplicity. We could allow more complicated potentials and, so long as we had $\V'(0) >0$ and $\V''(0) \ne 0$, only minor changes to our results would occur.}.
 Lastly
$$
 \delta^+f(j) =f(j+1) - f(j) \mand \delta^-f(j) = f(j) - f(j-1)
$$
are the {\it right and left finite-difference operators}.

Models of this sort are ubiquitous in applications. A partial list: molecular dynamics,
lamination, nondestructive testing, vehicular traffic, granular media, metamaterials,  chemistry/biochemistry, and power generation \cite{pankov}. 
The system \eqref{FPUT} also plays a major role as a paradigm for the mathematical analysis of wave propagation---especially solitary waves---in nonlinear dispersive settings and it is the system's famous connection to the Korteweg-de Vries (KdV) equation wherein our interest lies.
Here are several important mathematical results about that connection:
\begin{itemize}
\item When $m(j)$ is constant, long-wavelength (say like $1/\ep$, where $0< \ep \ll 1$),
small-amplitude (order of $ \ep^2$) solutions are well-approximated over long time scales (order of $ 1/\ep^3$) by solutions of KdV equations. The relative $\ell^2$-error
 made by the approximation in this case is $\O(\ep)$. See \cite{zabusky-kruskal} for the earliest formal derivation and \cite{schneider-wayne} for
the first rigorous result.
\item The same sort of result holds when  $m(j)$ is $N$-periodic (that is, $m(j+N) = m(j)$ for all $j$). Indeed, the spring potentials may also be taken to be $N-$periodic (for instance, replace  $\V(q)$ with $\V_j(q) = \kappa(j) q^2/2 + \beta(j) q^3/3 $ where $\kappa(j)$
and $\beta(j)$ are $N-$periodic).
See~\cite{bruckner-etal,GMWZ}.
\end{itemize}

While there have been a few derivations of KdV  from random versions of the FPUT lattice previously (specifically \cite{iizuka,wadati}), all have been purely formal with no rigorous quantitative results. Even conjectures for the size of the error have been absent. We have been working for some time to remedy this.
In our article \cite{mcginnis-wright} we discovered that if the $m(j)$ are independent identically distributed (i.i.d.)~random variables then
 the accuracy of long-wavelength approximations is substantially diminished and, consequently, only shorter time scales and the linear problem (that is when $\V'(q) = q$) are within reach.  Precisely, we showed\footnote{We also allow for heterogeneity in $\V$ as well as in $m$ in \cite{mcginnis-wright}; our results apply to the case where $\V'(q)$ replaced by $\V'_j(q)= \kappa(j) q$ with $\kappa(j)$ another collection of i.i.d.~random variables.} that long-wavelength solutions (again like $1/\ep)$ converge almost surely and strongly, but rather slowly, to solutions of a wave equation on time scales on the order of $1/\ep$: the relative $\ell^2$-error 
made by the approximation is almost surely $\O(\sqrt{\ep \ln |\ln(\ep)|})$. Numerics indicate that 
our error estimate is close to sharp. In \cite{mcginnis}, McGinnis proved a similar result for a 2D lattice.

Furthermore, formal and numerical studies of random FPUT and other similar random lattice problems report that the waves in such systems experience a notable deterioration of their amplitude as time evolves (see, for instance, \cite{garnier-etal,flach-etal-1,flach-etal-2,martinez-etal}).
We have carried out our own simulations of the nonlinear problem \eqref{FPUT} with i.i.d.~random variables as coefficients in the long-wavelength/small amplitude regime.
These simulations demonstrate that 
 for time scales longer than $1/\ep$, solutions of \eqref{FPUT} attenuate substantially; KdV-like dynamics (namely, resolution into solitons) is not observed.  
 We include the results of our simulations below in Section \ref{numerics}, Figure \ref{avt}. 
{\it In short, we do not believe that when the coefficients are i.i.d.~random variables a KdV approximation is appropriate or possible.} 

However, there are more sorts of randomness than simply taking the coefficients to be i.i.d.. 
In this paper we consider the random case, but we restrict the randomness in such a way that we can prove a fully rigorous KdV approximation.
We believe that this is the first example of such a result involving randomness and nonlinear dispersive systems, though there are several earlier results which carefully derive---but do not fully justify---KdV  as an effective equation for the evolution of long water waves over randomly varying topography \cite{papa-rosales,craig-guyenne-sulem,craig-etal,craig-sulem}.

Here is our assumption on the 
the masses:
\begin{assumption}\label{mass assumption}
The masses are given by
\be\label{transparency} 
m(j) = 1 + \delta^+ \delta^- \zeta(j)
\ee
where  $\zeta(j)$, $j \in \Z$, are i.i.d.~random variables with zero mean, variance $\sigma^2$
and support contained in $(-1/4,1/4)$.
\end{assumption}

We refer to \eqref{transparency} as a {\it transparency condition} and
we call \eqref{FPUT} subject to Hypothesis~\ref{mass assumption} the {\it transparent random mass FPUT lattice.} The use of ``transparent'' here is due to an observation from simulations: if the masses meet this condition then waves propagate relatively cleanly through the lattice without too much ``back scattering'' or ``internal reflection."  Our idea for making this choice was inspired by the derivation of KdV as an effective equation for  water waves over a random bottom in \cite{papa-rosales} where the topography is given as a perfect spatial derivative. The condition on the support of $\zeta(j)$ is there to ensure that $m(j)$ are strictly positive (for if $|\zeta(j)| <1/4$ then the triangle inequality tells us $m(j)>0$). It also guarantees that $\sigma^2<\infty$.

Our main result in a nutshell: 
for suitable initial conditions, solutions of the transparent random mass FPUT lattice almost surely satisfy
\be\label{nutshell}
\begin{split}
q(j,t) & = \ep^2\left[ A(\ep(j-t),\ep^3t) +B(\ep(j+t),\ep^3t) \right]+\O_{\ell_2}(\ep^{2} \sqrt{|\ln(\ep)|})\\
p(j,t) & =\ep^2\left[ -A(\ep(j-t),\ep^3t) +B(\ep(j+t),\ep^3t) \right]+ \O_{\ell_2}(\ep^{2}\sqrt{|\ln(\ep)|})
\end{split}
\ee
for $|t| \le T_0/\ep^3$, where $A$ and $B$ solve the KdV equations
$$
2 \partial_T A + \left( {1 \over 12} +2 \sigma^2 \right)  \partial_w^3 A + \partial_w A^2 = 0
\mand
2 \partial_T B - \left( {1 \over 12} +2 \sigma^2 \right)  \partial_l^3 B -\partial_l B^2= 0.
$$
The fully technical version of our result appears in Theorem \ref{main theorem} below.
\begin{remark}
To the uninitiated, it may look like the size of the error exceeds the size of the  approximation. However, the long-wave scaling of the spatial coordinate gives $\|\ep^2A(\ep(\cdot-t),\ep^3t)\|_{\ell^2} = \O 
(\ep^{3/2})$ which indicates a relative $\ell^2$-error of $\O\left(\sqrt{\ep |\ln(\ep)|}\right)$. 
\end{remark}


Our paper is organized as follows. Section \ref{preliminaries} spells out some notation and other ground rules. Section \ref{approximation section} proves a general approximation theorem for \eqref{FPUT}. While motivated by KdV approximations, the result applies more broadly. Section \ref{derivation} contains the derivation of KdV from \eqref{FPUT} under Hypothesis \ref{mass assumption}; this is the heart of the paper. Section \ref{estimates} contains a multitude of estimates which set the stage for the application of the general approximation theorem. It is in this section where probability plays a major role and where the technical guts of our result live. Section \ref{main event} ties everything together with the statement and proof of  our main result, the technical version of \eqref{nutshell}. 
Then we present the result of supporting numerics in Section \ref{numerics}.
We close out with a big list of open questions in Section \ref{next}.

\noindent{\bf Acknowledgements:} The authors would like to thank Amanda French,  C.~Eugene Wayne and Atilla Yilmaz for helpful conversations related to this project. Also, JDW would like to recognize the National Science Foundation who supported this research with grant DMS-2006172.
 
\section{Preliminaries}\label{preliminaries}
\subsection{Function/sequence spaces}
For a doubly infinite sequence $f:\Z \to \R$ we put, as per usual,
$ \| f\|_{\ell^2} := \sqrt{\sum_{j \in \Z} f^2(j) }$ and
$ \| f\|_{\ell^\infty} := \sup_{j \in \Z} |f(j)|$. Of course $\ell^2$ and $\ell^\infty$ are the sets of all sequences where the associated norms are finite. If we write $\|f,g\|_{\ell^2}$ we mean $\|f\|_{\ell^2} + \|g\|_{\ell^2}$, the norm on $\ell^2 \times \ell^2$. The analogous convention applies to $\|f,g\|_{\ell^\infty}$.
For functions $F: \R \to \R$ and non-negative integers $n$ and $r$ we put
$$
\| F \|_{H^n(r)}:=\sqrt{ \int_\R (1+X^2)^r F^2(X) dX +  \int_\R (1+X^2)^r (\partial_X^nF)^2(X) dX}
$$
and $H^n(r)$ is the closure of the set of all smooth functions with respect to this norm.
 We define $H^n:=H^n(0)$, $L^2(r):=H^0(r)$ and $L^2:=H^0(0)$. Next,  $\| F\|_{W^{n,\infty}}:=\sup_{X \in \R} |F(X)|+|\partial_X^nF(X)|$ and $W^{n,\infty}$ is the associated function space. By $L^\infty$ we mean $W^{0,\infty}$.
All of the spaces listed above are Banach spaces.

\subsection{Probability}
All probabilistic components in the paper descend through the random variables $\zeta(j)$. Associated probabilities are represented by $\P$ and expectations by $\Ex$.

\subsection{$\O$, $o$ and $C$ notation}
We use the following version of Landau's ``big $\O$/little $o$'' notation.
Given two real-valued functions, $f(\ep)$ and $g(\ep)$, we say $f(\ep) = \O(g(\ep))$  if for some $C_\star>0$ and $\ep_\star>0$, $|f(\ep)| \le C_\star |g(\ep)|$ for $\ep \in (0,\ep_\star)$. We say
$f(\ep) = o(g(\ep))$ if $\lim_{\ep\to0^+} |f(\ep)/g(\ep)| = 0$.
If $Y$ is a Banach space and we have a family of elements $u_\ep \in Y$, we write $u_\ep = \O_Y(g(\ep))$ 
if
$\|u_\ep\|_{Y} = \O(g(\ep))$ 
by the earlier definition. 
If we write $f(\ep) = g(\ep) + \O(h(\ep))$ (or similar) we mean $f(\ep) - g(\ep) = \O(h(\ep))$.
We at times default to ``big $C$'' notation: if we simply write $f(\ep) \le C g(\ep)$ and omit qualifiers 
then we mean $f(\ep) = \O(g(\ep))$. 

Some quantities will depend on the random variables $\zeta(j)$. 
For such quantities, if we write $f(\ep) = \O(g(\ep))$ we mean this in an almost sure sense. 
Specifically, there exist constants $C_\star>0$ (almost surely finite) and $\ep_\star$ (almost surely positive)  such that $|f(\ep)| \le C_\star |g(\ep)|$ for $\ep \in (0,\ep_\star)$. 
The values of $C_\star$ and $\ep_\star$ may depend upon the realization of the $\zeta(j)$.
The same ``almost sure'' point of view holds for $\O_Y$ and $o$ too.

To be clear: we always mean $\O$, $o$ and their ilk rigorously, and we always mean them in the almost sure sense.

\section{Approximation in general}\label{approximation section}
We begin by proving a general approximation theorem
 using the strategy
described in Section 5.3 of \cite{GMWZ} (itself inspired by \cite{schneider-wayne}). 
Suppose that
for $\ep \in (0,1)$ we have some functions $\tilde{q}_\ep(j,t)$ and $\tilde{p}_\ep(j,t)$ (the {\it approximators}) that we expect
are good approximations to solutions of \eqref{FPUT} when $\ep$ is small. By this we mean that we know that 
$\tilde{q}_\ep(j,t)$ and $\tilde{p}_\ep(j,t)$ nearly solve \eqref{FPUT} in the sense that the {\it residuals}
\be\label{res def}
\Res_1:=  \delta^+ \tilde{p}_\ep -\partial_t \tilde{q}_\ep \mand
\Res_2:= {1 \over m} \delta^-\left[ \V'(\tilde{q}_\ep) \right]- \partial_t \tilde{p}_\ep
\ee
are small relative to $\ep$.
To validate the approximation over the time scale $|t| \le T_0 /\ep^{3}$ 
 we need information about
\be\label{alphas}\begin{split}
\alpha_1(\ep)&:=\sup_{|t|\le T_0/\ep^3} \|\tilde{q}_\ep,\tilde{p}_\ep\|_{\ell^2}, \\
\beta_1(\ep)&:=\inf_{|t| \le T_0/\ep^3} \|\tilde{q}_\ep,\tilde{p}_\ep\|_{\ell^2},\\
\alpha_2(\ep)&:= \sup_{|t|\le T_0/\ep^3}  \|\partial_t \tilde{q}_\ep\|_{\ell^\infty}  \mand\\
\alpha_3(\ep)&:= \sup_{|t| \le T_0/\ep^{3}} \|\Res_1\|_{\ell^2} + \|\Res_2\|_{\ell^2}.\end{split}\ee
In particular, we assume:
\be\label{alf}
\alpha_1(\ep) = o(1),\quad \alpha_2(\ep) = \O(\ep^3)\mand \alpha_3(\ep) = o (\beta_1(\ep) \ep^3).
\ee

Our goal is to show that if we have approximators with these features then the true solution of \eqref{FPUT} whose initial conditions are consistent with those of the approximators remains close  over the long time scale. The result we prove here is specialized to FPUT lattices where the spring potentials are homogeneous and of the form \eqref{homog spring},  but 
requires only the following non-degeneracy condition on the masses:
\be\label{mass ass}
\inf_{j \in \Z} m(j) > 0 \mand \sup_{j \in \Z} m(j) < \infty.
\ee
The condition on the support of $\zeta(j)$ in Hypothesis \ref{mass assumption} implies the above, though we do not require all of that hypothesis in this section.

Here is the result:
\begin{theorem} \label{approx thm}
Suppose that the mass coefficients satisfy \eqref{mass ass}, the approximators $\tilde{q}_\ep(j,t)$ and $\tilde{p}_\ep(j,t)$ meet \eqref{alf}
and the initial conditions for \eqref{FPUT} satisfy
$$
\|q(0)-\tilde{q}_\ep(0),p(0)-\tilde{p}_\ep(0)\|_{\ell^2} =\O\left({\alpha_3(\ep) \over \ep^{3}}\right).
$$
Then the solution $(q(t),p(t))$ of \eqref{FPUT} satisfies the {\bf absolute error estimate}
$$
\sup_{|t|\le T_0/\ep^3} \|q(t)-\tilde{q}_\ep(t),p(t)-\tilde{p}_\ep(t)\|_{\ell^2} = \O\left({\alpha_3(\ep) \over \ep^{3}}\right)
$$
as well as the {\bf relative error estimate}
$$
\sup_{|t|\le T_0/\ep^3} {\|q(t)-\tilde{q}_\ep(t),p(t)-\tilde{p}_\ep(t)\|_{\ell^2} \over 
 \|q(t),p(t)\|_{\ell^2}} = \O\left({\alpha_3(\ep)\over\beta_1(\ep) \ep^3}\right).
$$
\end{theorem}

\begin{proof}
We introduce the {\it errors}:
$$
\eta:=q-\tilde{q}_\ep \mand \xi := p - \tilde{p}_\ep
$$
where $q(j,t)$ and $p(j,t)$ solve \eqref{FPUT}. 
Time differentiation of these expressions together with \eqref{FPUT} and some algebra get us:
\be\label{err}
\dot{\eta} =  \delta^+ \xi + \Res_1 \mand \dot{\xi} = {1 \over m}\delta^{-}(\W'(\eta,\tilde{q}_\ep)) 
+\Res_2.
\ee 
In the above
$$
\W(a,b):= \U(b+a)-\U(b)-\U'(b)a = {1 \over 2} (1+2 b) a^2 +{1 \over 3}a^3
$$
so that
$$
\W'(a,b):=\partial_a \W(a,b) = \U'(b+a)-\U'(b) = (1+ 2 b) a + a^2.
$$

Now we define the {\it energy} functional:
$$
\H(u,v;t):=\sum_{j \in \Z}\left( {1 \over 2} m_j v(j)^2 + \W(u(j),\tilde{q}_\ep(j,t))\right).
$$
In the above $(u,v)=(u(j),v(j))$ is  in $\ell^2\times \ell^2$.
Under our assumptions,
the square root of this quantity is equivalent to the $\ell^2 \times \ell^2$ norm in the following sense: 
 there exist $\ep_{*} \in (0,1)$ and  $C_*>1$ such that
\be\label{equiv}\begin{split}
&0 < \ep < \ep_{*},\ \|u\|_{\ell^2} \le 1,\ |t| \le T_0/\ep^3\\ \implies &
C_*^{-1} \| u,v\|_{\ell^2} \le \sqrt{\H(u,v;t)} \le C_*  \| u,v\|_{\ell^2}. 
\end{split}\ee

Here are the details. First of all, simple estimation gives
$$
{1 \over 2} \inf_{j \in \Z} m(j) \| v\|_{\ell^2}^2 \le  \sum_{j \in \Z} {1 \over 2} m_j v(j)^2 \le
{1 \over 2} \sup_{j \in \Z} m(j) \| v\|_{\ell^2}^2 
$$
and thus \eqref{mass ass} tells us that $\sqrt{   \sum_{j \in \Z} {1 \over 2} m_j v(j)^2}$ is equivalent to $\|v\|_{\ell^2}$. This gives the ``$v$'' part of \eqref{equiv}.

For the ``$u$'' part,  recall that $\| f \|_{\ell^\infty} \le \| f\|_{\ell^2}$ and so the assumption $\alpha_1(\ep) = o(1)$ implies that $\|\tilde{q}_\ep\|_{\ell^\infty} = o(1)$ as well. Thus for $\ep$ sufficiently small we have $\|\tilde{q}_\ep\|_{\ell^\infty} \le 1/30$.
This, in conjunction with the assumption that $\|u\|_{\ell^2} \le 1$, gives us
$$
\left \vert \sum_{j \in Z} \left( \tilde{q}_\ep(j,t) +{1 \over 3} u(j) \right) u(j)^2  \right \vert\le {11 \over 30} \|u \|_{\ell^2}^2.
$$
In turn the triangle inequality gives:
$$
{2 \over 15} \|u\|_{\ell^2}^2 \le 
\sum_{j \in Z}\bunderbrace{{1 \over 2} (1+2\tilde{q}_\ep(j,t))) u^2(j) + {1 \over 3} u^3(j)}
{\W(u(j),\tilde{q}_\ep(j,t))} \le {13 \over 15} \|u\|_{\ell^2}^2.
$$
So we have \eqref{equiv}.

For the next step, we suppose that $\eta(j,t)$ and $\xi(j,t)$ solve \eqref{err} and put
$H(t):=\H(\eta(t),\xi(t);t)$.
Differentiation of $H(t)$ with respect to $t$ gives:
$$
\dot{H}(t)=\sum_{j \in \Z}\left( m_j \xi(j,t)\dot{\xi}(j,t) + \W'(\eta(j,t),\tilde{q}_\ep(j,t) )\dot{\eta}(j,t) + {\partial_b \W}(\eta(j,t),\tilde{q}_\ep(j,t)) \partial_t\tilde{q}_\ep(j,t)
\right).
$$
Using \eqref{err} (and suppressing some dependencies) results in:
$$
\dot{H}=\sum_{j \in \Z}\left(  \xi  (\delta^{-}(\W'(\eta,\tilde{q}_\ep))+\Res_2) + \W'(\eta,\tilde{q}_\ep )(\delta^+ \xi + \Res_1) + {\partial_b \W}(\eta,\tilde{q}_\ep) \partial_t\tilde{q}_\ep
\right).
$$
We sum by parts and terms cancel:
$$
\dot{H}=\sum_{j \in \Z}\left(  \xi \Res_2 + \W'(\eta,\tilde{q}_\ep )\Res_1 + {\partial_b \W}(\eta,\tilde{q}_\ep) \partial_t\tilde{q}_\ep
\right).
$$
Subsequently, Cauchy-Schwarz and the like get us:
$$
\dot{H}\le \| \xi \|_{\ell^2} \|\Res_2\|_{\ell^2 }+ \| \W'(\eta,\tilde{q}_\ep )\|_{\ell^2}\|\Res_1\|_{\ell^2}+\|{\partial_b \W}(\eta,\tilde{q}_\ep)\|_{\ell^1} \| \partial_t \tilde{q}_\ep\|_{\ell^\infty} .
$$
One easily computes that
$
{\partial_b \W}(\eta,\tilde{q}_\ep)=\eta^2.
$
In which case we conclude, using the earlier formula for $\W'$ and routine estimates, that
\bes\label{E est 1}
\dot{H} \le \| \xi \|_{\ell^2} \|\Res_2\|_{\ell^2 }
+ \left(\left(1+2\| \tilde{q}_\ep\|_{\ell^\infty}\right)\|\eta\|_{\ell^2} +\|\eta\|_{\ell^2}^2\right)\|\Res_1\|_{\ell^2}
+\|\eta\|_{\ell^2}^2 \| \partial_t \tilde{q}_\ep\|_{\ell^\infty}.
\ees

Next we use \eqref{alf} to get:
\bes\label{E est 2}
\dot{H} \le 2 \alpha_3(\ep) \left( \|\eta\|_{\ell^2}+ \|\xi\|_{\ell^2} \right)+2\alpha_2(\ep) \|\eta\|_{\ell^2}^2.
\ees
Then we use \eqref{equiv}:
\bes\label{E est 3}
\dot{H} \le 2C^2_*\left(\alpha_3(\ep) H^{1/2} +\alpha_2(\ep) H\right).
\ees
 Since 
$\dot{H} = {2} H^{1/2} {d \over dt} H^{1/2}$ the above can be recast as
$$
{d \over dt} {H^{1/2}} \le C_*^2 \left(\alpha_3(\ep) +\alpha_2(\ep) H^{1/2}\right).
$$
We have assumed $\alpha_2(\ep) = \O(\ep^3)$ so the above implies
$$
{d \over dt} {H^{1/2}} \le C_*^2 \left(\alpha_3(\ep) +C_2 \ep^3 H^{1/2}\right)
$$
for a constant $C_2>0$.

An application of Gr\"onwall's inequality gets us:
$$
H^{1/2}(t) \le C_2^{-1} \left(e^{C_*^2C_2 \ep^3 t} - 1\right) {\alpha_3(\ep) \over \ep^3} + e^{C_*^2C_2 \ep^3 t} {H^{1/2}(0)}.
$$
Using \eqref{equiv} again: 
\bes\label{final}
\| \eta(t),\xi(t)\|_{\ell^2} \le C_*^2 C_2^{-1}\left(e^{C_*^2C_2 \ep^3 t} - 1\right) {\alpha_3(\ep) \over  \ep^3}+ C_*^2 e^{C_*^2C_2 \ep^3 t}\| \eta(0),\xi(0)\|_{\ell^2}.
\ees
We take the supremum of this over $|t| \le T_0/\ep^3$ and get
\bes
\sup_{|t| \le T_0 /\ep^3}\| \eta(t),\xi(t)\|_{\ell^2} \le C_\star \left( {\alpha_3(\ep) \over \ep^3} + \| \eta(0),\xi(0)\|_{\ell^2} \right).
\ees
The constant $C_\star>0$ is independent of $\ep$.

In conclusion, if we assume that 
\bes\label{ic ass}
\| \eta(0),\xi(0)\|_{\ell^2} = \O\left({\alpha_3(\ep)\over\ep^3}\right)
\ees
then we have shown
$$
\sup_{|t| \le T_0 /\ep^3}\| \eta(t),\xi(t)\|_{\ell^2} = \O\left({\alpha_3(\ep)\over\ep^3}\right).
$$
This is the absolute error estimate.
As for the relative error
%
a standard reverse triangle inequality argument shows that 
$$
\sup_{|t| \le T_0 /\ep^3}{ \| \eta(t),\xi(t)\|_{\ell^2} \over \|q(t),p(t)\|_{\ell^2}}= \O\left({\alpha_3(\ep)\over\beta_1(\ep)\ep^3}\right).
$$

\end{proof}

\section{Derivation of the effective equations}\label{derivation}
Now that we have Theorem \ref{approx thm}, we can move on to deriving the KdV equations from \eqref{FPUT}. 
The procedure for the derivation is a multiple scales expansion, inspired by \cite{donato}. We 
 assume the following form of our approximators:
\be\label{start}\begin{split}
\tilde{q}_\ep (j,t)  &= \sum_{n=0}^3 \ep^{n+2} Q_n(j,\ep j, \ep t, \ep^3 t)\mand
\tilde{p}_\ep (j,t)  = \sum_{n=0}^3 \ep^{n+2} P_n(j,\ep j, \ep t, \ep^3 t)
\end{split}\ee
where the $Q_n=Q_n(j,X,\tau,T)$ and $P_n=P_n(j,X,\tau,T)$ are maps 
$$
\Z \times \R \times \R \times \R \to \R.
$$
Of course we are viewing $\ep$ as being small.
Given that we put $X = \ep j$ in $\tilde{q}_{\ep}$ and $\tilde{p}_\ep$, we think of $X$ as being the {\it long-wave length scale} and $j$ being the {\it microscopic length scale.}

For expansions of the sort we are carrying out, it pays to be organized at the outset.
First we define the following operators for functions $U=U(j,X)$:
\bes\label{defs}\begin{split}
S^\pm_j U(j,X) := U(j\pm1,X) \mand
\delta_j^\pm U(j,X) :=  \pm \left(U(j\pm1,X) - U(j,X)\right).
\end{split}\ees
These are {\it partial shifts} and {\it partial finite-differences}  with respect to $j$.
Next, for $\ep >0$ put
\bes\begin{split}
D^+ U(j,X) := \pm(U(j\pm1,X\pm\ep)-U(j,X)).
\end{split}\ees
If $u(j) = U(j,\ep j)$ then $\delta^\pm u(j) = D^\pm U (j,\ep j).$ That is to say, $D^\pm$ are the {\it total finite-difference operators}.

Expanding the right-hand sides of $D^\pm U(j,X)$ in  (formal) Taylor series with respect to $\ep$ gives
$\ds
D^\pm U(j,X)  = \delta_j^\pm U(j,X) + \sum_{n=1}^\infty \ep^n{(\pm 1)^{n+1}  \over n!}S_j^\pm \partial_X^n U(j,X).
$
Truncating the sum at $n=M$ would give a formal error on the order of $\ep^{M+1}$ and 
so we define 
$$
\ep^{M+1} E^\pm_{M} := D^\pm - \delta_j^\pm - \sum_{n = 1}^M \ep^n{(\pm 1)^{n+1}  \over n!} S_j^\pm\partial_X^n.
$$
These operators give the exact error made by such a truncation. Note that for $M=0$ we just ignore the sum, {\it i.e.}~$\ep E_0^\pm := D^\pm-\delta_j^\pm$.

If we plug \eqref{start} into the residuals \eqref{res def} and carry out some substantial algebra we find 
that
\bes\begin{split}
\Res_1  &= \ep^2 Z_{12} + \ep^3 Z_{13} + \ep^4 Z_{14} + \ep^5 Z_{15} + \ep^6 W_1 \mand \\
\Res_2 &={1 \over m} \left( \ep^2 Z_{22} + \ep^3 Z_{23} + \ep^4 Z_{24} + \ep^5 Z_{25} + \ep^6 W_2\right)
\end{split}\ees
where
\bes\begin{split}
Z_{12}:=&\delta_j^+ P_0 \\
Z_{22}:=&\delta_j^- Q_0 \\
Z_{13}:=&\delta_j^+P_1 + \cSp \partial_X P_0 - \partial_\tau Q_0\\
Z_{23}:=&\delta_j^- Q_1 + \cSm  \partial_X Q_0 - m\partial_\tau P_0\\
Z_{14}:=&\delta_j^+P_2 + S_j^+\partial_X P_1 + {1 \over 2} S_j^+ \partial_X^2 P_0-\partial_\tau Q_1\\
Z_{24}:=& \delta_j^-Q_2 + S_j^-\partial_X Q_1 - {1 \over 2} S_j^- \partial_X^2 Q_0 + \delta_j^- Q_0^2- m\partial_\tau P_1\\
Z_{15}:=&\delta_j^+P_3 + S_j^+\partial_X P_2 + {1 \over 2} S_j^+ \partial_X^2 P_1+{1 \over 6} S_j^+ \partial_X^3 P_0\\
&-\partial_\tau Q_2-\partial_T Q_0 \\
Z_{25}:=&
 \delta_j^-Q_3 + S_j^-\partial_X Q_2 - {1 \over 2} S_j^- \partial_X^2 Q_1+{1 \over 6} S_j^- \partial_X^3 Q_0 \\&+ 2\delta_j^- (Q_0 Q_1) + \partial_X Q_0^2
-m\partial_\tau P_2-m\partial_T P_0 \\
\end{split}
\ees
and
\be\label{Ws}\begin{split}
W_1:=& \sum_{n=0}^3 E_{3-n}^+ P_n- \partial_\tau Q_3 - \sum_{n = 1}^3 \ep^{n-1} \partial_TQ_n
\\
W_2:=&  \sum_{n=0}^3 E_{3-n}^- Q_n + E_1^- Q_0^2 +2 E_0^- (Q_0 Q_1) \\
&+D^-\left(2 Q_0 Q_2 + \left(Q_1+\ep Q_2 + \ep^2 Q_3\right)^2
\right)
- m\partial_\tau P_3 - m\sum_{n = 1}^3 \ep^{n-1} \partial_TP_n.
\end{split}
\ee

The usual way to proceed is to select the $Q_0,P_0,\dots,Q_3,P_3$ so that
each $Z_{1k}=Z_{2k} = 0$. In this case we would have $\Res_1= \ep^6 W_1$ and $\ds \Res_2 = {1 \over m} \ep^6 W_2$ which we can then estimate using the formulas for the $Q_n$ and $P_n$. This strategy works perfectly well in the homogeneous and periodic problems as all the terms are rigorously the size they formally appear to be, modulo an annoying factor of $\ep^{-1/2}$ caused by the long-wave scaling.
But it fails in the random problem; the randomness leads to terms which are much larger  than they appear. 

Our modified strategy is to solve 
$$
Z_{12} = Z_{22} = Z_{13}=Z_{23} = Z_{14} = Z_{24} = 0
$$
(which will largely determine $Q_0,P_0,\dots,Q_2,P_2$) 
and then to do ``something else'' for $Z_{15}$ and $Z_{25}$.  
At the end of this, we find that
$
\Res_1= \ep^5 Z_{15} + \ep^6 W_1$ and $\ds \Res_2 = {1 \over m} \left(\ep^5Z_{25} + \ep^6W_2\right).
$
In  Section \ref{estimates} we show that these are $\O_{\ell^2} (\ep^5 \sqrt{|\ln(\ep)|})$. This is enough to apply Theorem \ref{approx thm} and get the error estimates shown in \eqref{nutshell}. 

\subsection{A tutorial on solving $Z_{1k} =Z_{2k}= 0$}\label{tutorial}
Each pair of equations $Z_{1k} =Z_{2k}= 0$ will have the form
\be\label{genZ}\begin{split}
\delta_j^+ P_{k-2} &= \Fb_0(X,\tau,T) + \sum_{n=1}^N f_n(j) \Fb_n(X,\tau,T)\\
\delta_j^- Q_{k-2} &= \Gb_0(X,\tau,T) + \sum_{n=1}^N g_n(j) \Gb_n(X,\tau,T).
\end{split}\ee
The sequences $f_n(j)$ and $g_n(j)$ are mean-zero random variables which come, in one way or another, from $\zeta(j)$; they depend only on the microscale coordinate. The $\Fb_n$ and $\Gb_n$ functions do not depend on the microscale coordinate at all. They will be made up of pieces of the various $P_n$ and $Q_n$ where $n < k - 2$. In this way  \eqref{genZ} allows us to figure out $P_{k-2}$ and $Q_{k-2}$ from the earlier functions.

We decompose \eqref{genZ} into a ``long-wave'' part (those pieces that do not depend on the microscale coordinate $j$ at all) and a ``microscale'' part (those that do). The long-wave part just consists of the terms $\Fb_0$ and $\Gb_0$ and so we set
\be\label{lwp}
\Fb_0 = 0 \mand  \Gb_0 =0.
\ee
This is a sort of solvability condition that will wind up giving us the long-wave dynamics; how it all plays out will be seen when we get in the weeds below. 

The microscale part is what is left over:
\be\label{msp}\begin{split}
\delta_j^+ P_{k-2} = \sum_{l=1}^N f_n(j) \Fb_n(X,\tau,T)\mand 
\delta_j^- Q_{k-2} =  \sum_{l=1}^N g_n(j) \Gb_n(X,\tau,T).
\end{split}\ee
We can just write down a solution for this:
\be\label{genZsol}\begin{split}
P_{k-2}(j,X,\tau,T) &= \Pb_{k-2}(X,\tau,T) +  \sum_{l=1}^N \chi_n(j) \Fb_n(X,\tau,T)
\mand\\
Q_{k-2}(j,X,\tau,T) &= \Qb_{k-2}(X,\tau,T) +  \sum_{l=1}^N \kappa_n(j) \Gb_n(X,\tau,T)
\end{split}\ee
where we select $\chi_n$ and $\kappa_n$ so that
\bes\label{grw}
\delta^+ \chi_n = f_n \mand \delta^- \kappa_n = g_n.
\ees
Solving these equations for $\chi_n$ and $\kappa_n$ from $f_n$ and $g_n$ is one of the key steps in the whole procedure and as we shall show the transparency condition makes this a relatively easy affair...at least at first. 
The functions $\Pb_{k-2}(X,\tau,T)$ and $\Qb_{k-2}(X,\tau,T)$ are  ``constants of integration''; in most cases we determine these from \eqref{lwp} at a later point in the derivation.

Now we get into actually solving the equations.

\subsection{$Z_{12} = Z_{22} = 0$}
These read
$
\delta_j^+ P_0 = 0$ and $\delta_j^- Q_0 = 0
$
which tells us that
\be\label{Q0P0}
Q_0(j,X,\tau,T) = \bar{Q}_0(X,\tau,T) \mand P_0(j,X,\tau,T) = \bar{P}_0(X,\tau,T).
\ee

\begin{remark}
In this section any function with a ``bar'' on top will not depend on $j$. We make this convention so that we do not need to perpetually clutter up our algebra with functional dependencies. For the same reason it is helpful to keep in mind that $m$ and $\zeta$ depend only on $j$ and not on the other variables.
\end{remark}

\subsection{$Z_{13} = Z_{23} = 0$}\label{Z13}
Using  \eqref{Q0P0} these equations become
\bes\label{z130}
\delta_j^+P_1=  \partial_\tau \bar{Q}_0 - \partial_X \bar{P}_0  \mand 
\delta_j^- Q_1 = m \partial_\tau \bar{P}_0 - \partial_X \bar{Q}_0. 
\ees
Using the transparency condition \eqref{transparency} converts these to
\bes\label{z13}
\delta_j^+P_1=  \partial_\tau \bar{Q}_0 - \partial_X \bar{P}_0  \mand 
\delta_j^- Q_1 =  \partial_\tau \bar{P}_0 - \partial_X \bar{Q}_0 + \delta^+ \delta^- \zeta  \partial_\tau \bar{P}_0.
\ees
Following the steps from the tutorial in Section \ref{tutorial} we see that the long-wave part \eqref{lwp} of these equations is
\be\label{wave}\begin{split}
  \partial_\tau \bar{Q}_0 - \partial_X \bar{P}_0 =0\mand
\partial_\tau \bar{P}_0 - \partial_X \bar{Q}_0=0.
\end{split}\ee
This is  the wave equation wearing a fake mustache and glasses and we readily solve it:
\be\label{AB}
\bar{Q}_0=A(X-\tau,T) + B(X+\tau,T) \mand \bar{P}_0 = -A(X-\tau,T) + B(X+\tau,T).
\ee
\begin{remark}
We use the convention that $$w= X- \tau \mand l = X+\tau$$ so that $A = A(w,T)$ and $B = B(l,T)$.
Note that $A$ and $B$ do not depend on $j$. It is these functions that will ultimately solve KdV equations. 
\end{remark}
After \eqref{wave} we are left with the microscale part
\bes\label{z132}
\delta_j^+P_1=  0  \mand 
\delta_j^- Q_1 =   \delta^+ \delta^- \zeta  \partial_\tau \bar{P}_0.
\ees
The solution formula \eqref{genZsol} gives
$Q_1 = \bar{Q}_1 + \chi \partial_\tau\Pb_0$ where we want $\delta^- \chi = \delta^+ \delta^- \zeta$. Finding $\chi$ is easily done as we can simply cancel a $\delta^-$ from both sides and put $\chi = \delta^+\zeta$. This is so simple  because of the transparency condition \eqref{transparency} and this is one of the reasons we have assumed it.
Likewise \eqref{genZsol} says that we should put $P_1 = \Pb_1$, but it will turn out that $\Pb_1$ will be zero so we just enforce that now. In short we have
\be\begin{split}\label{Q1P1}
Q_1 = \bar{Q}_1 + \delta_j^+ \zeta \partial_\tau \bar{P}_0 \mand
P_1 = 0.
\end{split}\ee
Note that $\delta_j^+\zeta$ is bounded in $j$ because of the compact support assumption in Hypothesis \ref{mass assumption}.

\begin{remark}\label{wi1} What if we had not made the transparency assumption but instead assumed that 
$m(j) = 1 + z(j)$ where $z(j)$ are i.i.d.~ mean zero random variables? The long-wave part is the same as above but now the microscale part is $\delta_j^- Q_1 =  z  \partial_\tau \bar{P}_0.$ To use the solution formula \eqref{genZsol} we would want to find $\chi$ so that $\delta^- \chi = z$,
or rather
\bes
\label{rw}
\chi(j) = \chi(j-1) + z(j).
\ees
This equation tells us that
$\chi(j)$ is a random walk with steps given by  $z(j)$ and as such we expect $\chi$ 
to grow like $\sqrt{j}$. 

To see why this is an issue, notice that $Q_1$ would include the term
$
\chi(j) A_w(X-\tau,T),
$
which  then would show up in the approximator \eqref{start} as
$
\ep^3\chi(j) A_w(\ep(j-t),\ep^3 t).
$
The term $A_w$ is propagating to the right with roughly unit speed and thus when $t \sim 1/\ep^3$ will be located at $j \sim 1/\ep^3$. In turn this indicates  $\chi(j) \sim \ep^{-3/2}$ towards the end of the approximation time interval. Hence the term $\ep^3\chi A_w$ would be substantially larger than it appears: the techniques from \cite{mcginnis-wright} show that almost surely
$$
\sup_{|t| \le T_0/\ep^{3}}\| \ep^3\chi(\cdot) A_w(\ep(\cdot-t),\ep^3 t)\|_{\ell^2}
\le C \ep \sqrt{\ln|\ln(\ep)|}.
$$
Were $\chi = \O_{\ell^\infty}(1)$ the right-hand side of the preceding estimate would be $C\ep^{5/2}$ (see Lemma \ref{lwa} below). And so we find that 
the ``$\ep^3$ term'' in the approximator is more than an order of magnitude larger than it should be, bigger in fact that the leading order term in the approximation. Disaster!
 
The lesson learned: if a term in our approximation involves a random walk it will ultimately be at least $\ep^{-3/2}$ larger than it formally appears to be. We call this difficulty  \bf{a random walk disaster}.
\end{remark}

\subsection{$Z_{14} = Z_{24} = 0$} 
The relations \eqref{transparency}, \eqref{Q0P0}, \eqref{wave},  \eqref{Q1P1}
and a little algebra convert these equations to
\bes\begin{split}
\delta_j^+ P_2 &= \partial_\tau \bar{Q}_1  - {1 \over 2} \partial_{X\tau}^2 \bar{Q}_0
+ \delta_j^+ \zeta\partial_X^2 \bar{P}_0\mand
\delta_j^- Q_2 = - \partial_X \bar{Q}_1 + {1 \over 2} \partial_X^2 \bar{Q}_0
 -\delta_j^- \zeta \partial^2_{X} \bar{Q}_0.
\end{split}\ees
The long-wave part of this is 
\bes\begin{split}
0 &= \partial_\tau \bar{Q}_1  -  {1 \over 2} \partial_{X\tau}^2 \bar{Q}_0
\mand 
0 = - \partial_X \bar{Q}_1 + {1 \over 2} \partial_X^2 \bar{Q}_0
\end{split}\ees
which can be solved by putting
\be\label{bQ1bP1}
\Qb_1 = {1 \over 2} \partial_X \Qb_0={1 \over 2} \left(\partial_wA +\partial_l B\right).
\ee

This leaves
 the microscale part
\be\begin{split}\label{dur}
\delta_j^+ P_2 &=  \delta_j^+ \zeta\partial_X^2 \bar{P}_0\mand
\delta_j^- Q_2 =
 -\delta_j^- \zeta \partial^2_{X} \bar{Q}_0.
\end{split}\ee
Which as per \eqref{genZsol} we solve by
taking
\be\begin{split}\label{Q2P2f}
P_2 &= \bar{P}_2 +\zeta \partial_{X}^2 \bar{P}_0 \mand
Q_2 = \bar{Q}_2 - \zeta \partial_{X}^2 \bar{Q}_0.
\end{split}\ee   

Once again, the transparency condition \eqref{transparency} made finding this solution a simple matter of cancelation; it is the reason why the transparency condition has two finite-differences on $\zeta$. If we had put only one finite-difference in \eqref{transparency} then another random walk disaster as described in Remark \ref{wi1} would occur when we solve \eqref{dur}.

\subsection{Something else for $Z_{15}$ and $Z_{25}$}
The relations \eqref{transparency}, \eqref{Q0P0},  \eqref{wave}, \eqref{Q1P1},
 \eqref{bQ1bP1},  \eqref{Q2P2f} and quite a bit of algebra get us:
\bes\label{Z5s}\begin{split}
Z_{15}=&-
\partial_T \Qb_0 
-  \partial_\tau  \Qb_2   + \partial_X \Pb_2 
  + {1 \over 6}  \partial_X^3\Pb_0\\&+\delta_j^+ P_3  +  \left( \zeta 
 +S_j^+ \zeta 
 \right)  \partial_X^3\Pb_0\\
Z_{25}=&    
-\partial_T \Pb_0 -  \partial_\tau  \Pb_2 +  \partial_X \Qb_2
 -\left( {1 \over 12}  -2 \sigma^2
\right) \partial_X^3 \Qb_0  +\partial_X(\Qb_0^2)
\\&  
+\delta_j^- Q_3-\left(\zeta  + S_j^-\zeta +{1 \over 2} \delta_j^- \zeta
+ \zeta {\delta_j^+ \delta_j^- \zeta}
 +2 \sigma^2
\right) \partial_X^3 \Qb_0 \\&
  +  \delta_j^- \delta_j^+ \zeta \partial_X (\Qb^2_0)- {\delta_j^+ \delta_j^- \zeta} \left(\partial_T \Pb_0 +  \partial_\tau  \Pb_2 \right).
\end{split}\ees
Recall that $
\sigma^2
$ is the variance of $\zeta(j)$. 
We need $Z_{15}$ and $Z_{25}$ to be small relative to $\ep$, but zeroing them out completely happens to be too restrictive; we will need to modify the microscale part of the decomposition described in Section \ref{tutorial}. 
But before we get there we deal with the long-wave part.

\subsubsection{Kill the long-wave part with KdV equations}
As per normal, we zero out the long-wave parts of $Z_{15}$ and $Z_{25}$. We have conveniently arranged all such terms in the first line on the right of the preceding formulas for $Z_{15}$ and $Z_{25}$ and so we put
\be\begin{split}\label{protoKdV}
0 & =  
-\partial_T \Qb_0 
-  \partial_\tau  \Qb_2   + \partial_X \Pb_2 + {1 \over 6} \partial_X^3\Pb_0\\
0& =  
-\partial_T \Pb_0 -  \partial_\tau  \Pb_2 +  \partial_X \Qb_2 
-\left( {1 \over 12}  -2 \sigma^2
\right) \partial_X^3 \Qb_0 + \partial_X(\Qb_0^2).
\end{split}\ee


Within \eqref{protoKdV} lurk the KdV equations; here is how we coax them into the daylight. Let 
\begin{equation*}\begin{split}
\Qb_2 (X,\tau,T) &= A_2(X-\tau,X+\tau,T) + B_2(X-\tau,X+\tau,T)\\\mand
\Pb_2 (X,\tau,T) &= -A_2(X-\tau,X+\tau,T) + B_2(X-\tau,X+\tau,T)
\end{split}\end{equation*}
and use \eqref{AB} in \eqref{protoKdV} to get
\bes\begin{split}
0  =&  
\partial_T A + \partial_T B
+  2 \partial_l A_2  - 2 \partial_w B_2
- {1 \over 6} (-\partial_w^3A + \partial_l^3 B)\\
0 =&  
-\partial_T A + \partial_T B   
-2 \partial_l A_2 - 2\partial_w B_2
+\left( {1 \over 12}  -2 \sigma^2
\right) (\partial_w^3A + \partial_l^3 B) \\&- \left(\partial_w A^2 + 2 \partial_w A B + 2 A \partial_l B + \partial_l B^2 \right).
\end{split}\ees
Subtracting these gives:
\be\label{almost there}\begin{split}
0  =&  
2 \partial_T A + \left( {1 \over 12} +2 \sigma^2 \right)  \partial_w^3 A + \partial_w A^2\\
 + &4 \partial_l A_2
-\left( {1 \over 4}  -2 \sigma^2
\right) \partial_l^3B  + \left(  2 \partial_w A B + 2 A \partial_l B
+\partial_l B^2  \right).
\end{split}\ee
If we let $\B$ be an $l$-antiderivative of $B$ (specifically $\B(l,T) := \int_0^l B(y,T) dy$) and set
 \be\label{A2}
   A_2 = {1 \over 4} \left[
   \left( {1 \over 4} - 2 \sigma^2\right) \partial_l^2B -
   \left(2 \partial_w A \B+ 2 A B + B^2
   \right)
   \right]
\ee
many terms in \eqref{almost there} die. What survives is
\be\label{KdVA}\begin{split}
0=2 \partial_T A + \left( {1 \over 12} +2 \sigma^2 \right)  \partial_w^3 A + \partial_w A^2.
\end{split}\ee
This is a KdV equation!
A parallel argument (after adding instead subtracting equations a few steps above) shows we should take 
\be\label{B2}
   B_2 = {1 \over 4} \left[
\left( {1 \over 4}  -2 \sigma^2
\right) \partial_w^2A  - \left( A^2 + 2 A B + 2 \A \partial_l B  \right)\right]\\
\ee
with $\A$ a $w$-antiderivative of $A$ (specifically $\A(w,T): = \int_0^w A(y,T) dy$).
In which case we get that $B$ solves another KdV equation:
\be\label{KdVB}\begin{split}
0=2 \partial_T B - \left( {1 \over 12} +2 \sigma^2 \right)  \partial_l^3 B -\partial_l B^2.
\end{split}\ee
To summarize: taking $A$, $B$, $A_2$ and $B_2$ as we have just described 
means that \eqref{protoKdV} is satisfied.

%
%
%

\subsubsection{Handle with the microscopic part using autoregressive processes.}
The next step in dealing with $Z_{15}$ and $Z_{25}$ is to control the microscopic parts that are left over after \eqref{protoKdV}:
\be\label{Z52}\begin{split}
Z_{15}=&\delta_j^+ P_3  +  \left( \zeta 
 +S_j^+ \zeta 
 \right)  \partial_X^3\Pb_0\\
Z_{25}=&\delta_j^- Q_3-\left(\zeta  + S_j^-\zeta +{1 \over 2} \delta_j^- \zeta
+ \zeta {\delta_j^+ \delta_j^- \zeta}
 +2 \sigma^2
\right) \partial_X^3 \Qb_0 \\&
  +  \delta_j^- \delta_j^+ \zeta \partial_X (\Qb^2_0)- {\delta_j^+ \delta_j^- \zeta} \left(\partial_T \Pb_0 +  \partial_\tau  \Pb_2 \right).
\end{split}\ee
Many, but not all, of these terms in $Z_{25}$ can be eliminated with the same cancelation tricks that worked earlier. To see this, we let
\bes\begin{split}
P_3&= \gamma_1 \partial_X^3 \Pb_0 \mand \\
 Q_3&=\left(\gamma_2 +{1 \over 2} \zeta\right)
 \partial_X^3 \Qb_0 
  - \delta_j^+ \zeta \partial_X (\Qb^2_0)+ {\delta_j^+ \zeta} \left(\partial_T \Pb_0 +  \partial_\tau  \Pb_2 \right)
\end{split}\ees
where for the moment we leave $\gamma_1 = \gamma_1(j)$ and $\gamma_2 = \gamma_2(j)$  
unspecified. Substituting the above into \eqref{Z52} gives
\be\label{Z525}\begin{split}
Z_{15}=&\left[\delta_j^+\gamma_1+\zeta 
 +S_j^+ \zeta 
 \right]  \partial_X^3\Pb_0\\
Z_{25}=&\left[\delta_j^- \gamma_2-\left(\zeta  + S_j^-\zeta 
+ \zeta {\delta_j^+ \delta_j^- \zeta}
 +2 \sigma^2
\right)\right] \partial_X^3 \Qb_0 .
\end{split}\ee

If we followed the strategy from the tutorial in Section \ref{tutorial}, we would put $\delta_j^+\gamma_1=-\zeta 
 -S_j^+ \zeta$ and
$ \delta^- \gamma_2=  \zeta+S^-\zeta + \zeta \delta^+ \delta^- \zeta+ 2 \sigma^2$ 
 and get $Z_{15}=Z_{25}=0$. Since the $\zeta(j)$ are i.i.d.~random variables we would then find that $\gamma_1(j)$ and $\gamma_2(j)$ are random walks, which leads us to another disaster as described in Remark \ref{wi1} (this time in the residual terms).
 Why not just stack another finite-difference on $\zeta$ in the transparency condition? This would help in $Z_{15}$ but would not be useful in handling the parts stemming from $\zeta {\delta_j^+ \delta_j^- \zeta}$ in $Z_{25}$.

To avoid these problematic random walks we take $\gamma_1$ and $\gamma_2$ to solve 
\be\begin{split}\label{AR1}
\delta^+ \gamma_1&= -\ep\sgn(j) \gamma_1 -\left( \zeta+S^+\zeta \right) \mand \\
\delta^- \gamma_2&= -\ep\sgn(j) \gamma_2+\left( \zeta+S^-\zeta + \zeta \delta^+ \delta^- \zeta+ 2 \sigma^2\right).
\end{split}\ee
In which case we find that \eqref{Z525} becomes
\bes\label{Z153}
Z_{15}= -\ep\sgn(j) \gamma_1\partial_X^3\Pb_0 \mand 
Z_{25} = -\ep \sgn(j) \gamma_2 \partial_X^3 \Qb_0.
\ees

The extra factors of $\ep$ on the right-hand sides here means that our choices of $P_3$ and $Q_3$ are formally as good as putting $Z_{15} = Z_{25} = 0$. 
Estimates for $P_3$ and $Q_3$ (and consequently $Z_{15}$, $Z_{25}$ and the residuals) ultimately require us to understand $\gamma_1$ and $\gamma_2$. The equations in \eqref{AR1} are examples of {\it autoregressive processes} \cite{grimmett-stirzaker}. These are dissipative cousins of random walks 
and with classical probabilistic methods we will show that they roughly cost us a factor of $\ep^{-1/2}$ (see Lemma \ref{gamma est} below) instead of the $\ep^{-3/2}$ we get from using random walks. This is big but not too big for our estimates to handle.

\subsection{Summing up.}
At this point we have completely determined all the functions $P_0,\dots,Q_3$ in the approximation. As it can be challenging to sort through it all,
we close out this section by summarizing the derivation.
\begin{definition}\label{full approx}
Suppose $A(w,T)$ and $B(l,T)$ solve the KdV equations \eqref{KdVA} and \eqref{KdVB} and  $\gamma_1(j)$ and $\gamma_2(j)$ solve the autoregressive processes \eqref{AR1}. Take $A_2(w,l,T)$ and $B_2(w,l,T)$ as in  \eqref{A2} and \eqref{B2}. Define $Q_k(j,X,\tau,T)$ and $P_k(j,X,\tau,T)$ via
\bes
{\begin{tabular} { l || l }

$Q_0 = A+ B$, & $P_0 = - A + B$\\
$Q_1 = {1 \over 2} \partial_X Q_0 + \delta_j^+ \zeta \partial_\tau P_0$, & $P_1 = 0$\\
$Q_2 = A_2 + B_2 - \zeta \partial_X^2 Q_0$ & $P_2=-A_2 + B_2 + \zeta \partial_X^2 P_0$\\
$\begin{aligned}
Q_3 = &\left(\gamma_2 +{1 \over 2} \zeta\right) \partial_X^3 Q_0 
         - \delta_j^+ \zeta \partial_X (Q^2_0)\\+& {\delta_j^+ \zeta} \left(\partial_T P_0 -  \partial_\tau  A_2 +\partial_\tau B_2\right)\end{aligned} $ & $\begin{aligned}P_3 = &\gamma_1 \partial_X^3 P_0\\ &\phantom{boo} \end{aligned}$\\

\end{tabular}}
\ees
where it is understood that $w = X-\tau$ and $l=X+\tau$. Then we
call $\tilde{q}_\ep(j,t)$ and $\tilde{p}_\ep(j,t)$, as defined in \eqref{start}, the {\bf extended KdV approximators}.
\end{definition}

In this section we have proven:
\begin{lemma}\label{residual calculation}
The extended KdV approximators 
have
$$
\Res_1= \ep^6\left(-\sgn(j) \gamma_1 \partial_X^3 P_0 + W_1\right)\mand
\ds \Res_2 = {\ep^6\over m} \left(- \sgn(j) \gamma_2 \partial_X^3 Q_0+ W_2\right)
$$
with $W_1$ and $W_2$ given at \eqref{Ws}.
\end{lemma}

We move on to proving many estimates related to the extended KdV approximators.

\section{Estimates on the approximators and residuals}\label{estimates}
%
%
%
To streamline some of the forthcoming statements we put forth the following convention:
\begin{definition}
We say $A$ and $B$ are {\bf good solutions of KdV on $[-T_0,T_0]$} if they satisfy \eqref{KdVA} and \eqref{KdVB} along with the estimate
$$
0<\sup_{|T| \le T_0} \|A(\cdot,T)\|_{H^7(1)} 
+ \|B(\cdot,T)\|_{H^7 (1)}  < \infty.
$$
\end{definition}

\begin{remark} The existence of good solutions of KdV on intervals of arbitrary length 
is by now classical (see \cite{tao}). 
The lower bound is just to guarantee that the approximation is not trivial.
\end{remark}

In this section we prove:
\begin{proposition}\label{big prop} 
Assume Hypothesis \ref{mass assumption}.
Let $\tilde{q}_\ep(j,t)$ and $\tilde{p}_\ep(j,t)$ be the {extended KdV approximators}
as in Definition \ref{full approx} where we further assume that that $A$ and $B$ are good solutions of KdV on $[-T_0,T_0]$.
Then almost surely the quantities defined at \eqref{alphas} satisfy
\bes\label{set yr goals}
\alpha_1(\ep) = \O(\ep^{3/2}),\
\alpha_2(\ep) = \O(\ep^3), \ \alpha_3(\ep) = \O(\ep^5 \sqrt{|\ln(\ep)|})\ {\text{ and }}\ \beta^{-1}_1(\ep) = \O(\ep^{-3/2}).
\ees
\end{proposition}


Estimates on terms which do not involve $\gamma_1$ or $\gamma_2$ can be handled using well-understood techniques found in previous works, whereas the rest require new ideas. 
All dependence on $\gamma_1$ and $\gamma_2$ enters through $P_3$ and $Q_3$, the latter of which has some terms without them. And so we put
\be\label{Qgamma}
Q_{3\gamma}:=\gamma_2 \partial_X^3 Q_0 \mand 
Q_{30} := Q_3 - \gamma_2 \partial_X^3Q_0.
\ee
To be clear, $Q_{30}$ has no instances of a $\gamma$ within.

Similarly if in the formulas for $W_1$ and $W_2$ we eliminate any term with a $\gamma$ in it we get:
\bes\begin{split}
W_{10}:=& \sum_{n=0}^2 E_{3-n}^+ P_n - \sum_{n = 1}^2 \ep^{n-1} \partial_TQ_n- \ep^{2} \partial_TQ_{30}
\\
W_{20}:=&  \sum_{n=0}^2 E_{3-n}^- Q_n + E_0^- Q_{30}+ E_1^- Q_0^2 +2 E_0^- (Q_0 Q_1) \\
&+D^-\left(2 Q_0 Q_2 + \left(Q_1+\ep Q_2+\ep^2Q_{30}\right)^2
\right)
 - m\sum_{n = 1}^2 \ep^{n-1} \partial_TP_n.
\end{split}\ees
Thus the terms with a $\gamma$ are:
\be\label{wt}
W_{1\gamma} := W_1-W_{10} \mand W_{2\gamma}:=W_2-W_{20}.
\ee

\subsection{Terms without $\gamma_1$ and $\gamma_2$}
In this part we prove:
\begin{lemma}\label{res prop1} Assume Hypothesis \ref{mass assumption}.
Let $\tilde{q}_\ep(j,t)$ and $\tilde{p}_\ep(j,t)$ be the {extended KdV approximators}
as in Definition \ref{full approx} where we further assume that that $A$ and $B$ are good solutions of KdV on $[-T_0,T_0]$.
Then 
$$
\sum_{n=0}^2\sup_{|t| \le T_0/ \ep^{3}} 
\left(\|P_n(\cdot,\ep \cdot, \ep t, \ep^3 t)\|_{\ell^2} +
\|Q_n(\cdot,\ep \cdot, \ep t, \ep^3 t)\|_{\ell^2} \right)= \O(\ep^{-1/2}),
$$
$$
\sup_{|t| \le T_0 /\ep^{3}} 
\|Q_{30}(\cdot,\ep \cdot, \ep t, \ep^3 t)\|_{\ell^2}= \O(\ep^{-1/2}),
$$
and
$$
\sup_{|t| \le T_0 /\ep^{3}} \left(
\|W_{10}(\cdot,\ep \cdot, \ep t, \ep^3 t)\|_{\ell^2}+
\|W_{20}(\cdot,\ep \cdot, \ep t, \ep^3 t)\|_{\ell^2} \right)= \O(\ep^{-1/2}).
$$
\end{lemma}
\begin{remark} Note that in Hypothesis \ref{mass assumption} we assumed that $|\zeta(j)|<1/4$ for all $j$.
A consequence of this is that none of the estimates in Lemma \ref{res prop1} depend on the realization of the $\zeta(j)$. That is to say there is no probability needed to understand this lemma.
\end{remark}

\begin{proof} The proof is similar to that of Proposition 4.2 of \cite{GMWZ}, though there are a few small, but substantive, differences. The main tool we need is:
\begin{lemma} \label{lwa}
Let $M \ge 0$ be an integer. 
Suppose that $f (j)\in \ell^\infty$ and $F(X) \in H^{M+1}$. If $u_\ep(j) := f(j) F(\ep j)$
then
$$
\|u\|_{\ell^\infty} \le \|f\|_{\ell^\infty} \|F\|_{L^\infty},
$$
$$
\|u\|_{\ell^2} \le C \ep^{-1/2} \|f\|_{\ell^\infty} \|F\|_{H^1},
$$
$$
\|E_M^\pm u\|_{\ell^2} \le C \ep^{-1/2} \|f\|_{\ell^\infty} \|F\|_{H^{M+1}}
$$
and
$$
\|D^\pm u\|_{\ell^2} \le C \ep^{-1/2}\|f\|_{\ell^\infty} \|F\|_{H^{1}}.
$$
The constants $C>0$ depend only on $M$.
\end{lemma}

\begin{proof}  
Lemma 4.3 of \cite{GMWZ} is nearly identical to this, but has the requirement that $f(j)$ be $N$-periodic. Still we can piggyback the proof of our result on that one. 
The first estimate is all but obvious. 
For the second we have the easy estimate $\|u \|_{\ell^2} \le \|f\|_{\ell^\infty} \|F_\ep\|_{\ell^2}$ 
where $F_\ep(j) := F(\ep j)$, $j \in \Z$. But then the second estimate of Lemma 4.3 of \cite{GMWZ} applies and shows $\|F_\ep\|_{\ell^2}\le C \ep^{-1/2}\|F\|_{H^1}$.  For the third, a direct computation shows that
$\ds
E_M^+ u (j) = f(j+1) (E_M^+ F_\ep)(j)
$
which implies $\|E_M^+ u\|_{\ell^2} \le \|f\|_{\ell^\infty}\|E_M^+ F_\ep\|_{\ell^2}$. 
The third estimate from Lemma 4.3 of \cite{GMWZ} implies that $\|E_M^+ F_\ep\|_{\ell^2} \le C \ep^{-1/2} \|F\|_{H^{M+1}}$. 
To get an estimate for $E_M^-$ is similar. The final estimate, for $D^\pm u$, follows from, the definition of $D^\pm$, the triangle inequality and the second estimate in this lemma.

\end{proof}

We also need the following, to control the antiderivatives in $A_2$ and $B_2$:
\begin{lemma} \label{anti-est}
Suppose that $F(X) \in L^2(1)$ then $\mathcal{F}(X) := \int_0^X F(y) dy$ is in $L^\infty$ and  $\| \mathcal{F} \|_{L^\infty} \le \sqrt{\pi}\|F\|_{L^2(1)}$. 
\end{lemma}
\begin{proof} We use Cauchy-Schwarz and the fact that $\int_\R (1+y^2)^{-1} = \pi $. To wit:
\begin{equation*}
\begin{split}
\left \vert \mathcal{F}(X) \right \vert 
\le 
& \int_0^X (1+y^2)^{-1/2} (1+y^2)^{1/2} | F(y)| dy\\
\le & \sqrt{\int_\R (1+y^2)^{-1}dy} \sqrt{\int_\R (1+y^2) | F(y)|^2 dy}
= \sqrt{\pi} \|F\|_{L^2(1)}.
\end{split}
\end{equation*}
Taking the supremum over $X$ seals the deal.
\end{proof}

Armed with Lemmas \ref{lwa} and \ref{anti-est} we can get into proving the estimates in the Lemma \ref{res prop1}. There are many terms and handling each would inflate this paper like a bounce house. 
So we do not do that. Instead we show how to estimate a ``prototype'' term which captures  the nuances. That term is
$$
g=E_0^-\left( \delta_j^+ \zeta \A \partial_l^2 B \right)
$$
which some digging will show appears in $\Res_2$. Using the estimate for $E_0^-$ from Lemma \ref{lwa} we have
$$
\|g\|_{\ell^2} \le C\ep^{-1/2} \| \delta^+_j \zeta\|_{\ell^\infty} \| \A \partial_l^2 B\|_{H^1}.
$$
By the triangle inequality and the definition of $\delta_j^+$ we have $\|\delta^+_j \zeta\|_{\ell^\infty} \le 2 \| \zeta\|_{\ell^\infty}$ and the supposition that the support of $\zeta(j)$ is in $(-1/4,1/4)$ ultimately gives $\|\delta^+_j \zeta\|_{\ell^\infty} \le 1/2$.
Also classical Sobolev-H\"older inequalities tell us that 
$\| \A \partial_l^2 B\|_{H^1} \le \| \A \|_{W^{1,\infty}} \|\partial_l^2B \|_{H^1} \le \| \A \|_{W^{1,\infty}} \|B \|_{H^3}$. 

Since $\| \A \|_{W^{1,\infty}} = \| \A \|_{L^\infty} + \| \A_w\|_{L^\infty}$ and $\A$ is an antiderivative of $A$ we can use Lemma \ref{anti-est} to conclude that  $\| \A \|_{L^\infty} \le \sqrt{\pi} \|A\|_{L^2(1)}.$ Likewise, Sobolev's inequality tells us that $\| \A_w\|_{L^\infty} = \|A\|_{L^\infty} \le C \|A\|_{H^1}$. So all together we have
$$
\|g\|_{\ell^2} \le C \ep^{-1/2} \left(  \|A\|_{L^2(1)}+\|A\|_{H^1}\right) \|B\|_{H^3}.
$$
Since $A$ and $B$ are assumed to be good solutions of KdV on $[-T_0,T_0]$ we get
$
\sup_{|t|\le T_0/\ep^3} \|g\|_{\ell^2} \le C \ep^{-1/2},
$
which is the targeted estimate.

All the other terms are handled using the same sorts of steps used above. We close the proof with a comment on the regularity needed. 
The most smoothness  required for $A$ and $B$ comes from the terms in $\partial_T Q_{30}$. As in \cite{schneider-wayne,bruckner-etal,GMWZ}, one finds that $\partial_w^6 A$ and $\partial_l^6B$ make an appearance and so, to deploy estimates like in Lemma \ref{lwa}, we need $A$ and $B$ to be in $H^7$.
\end{proof}

\subsection{The autoregressive part}
Now we need to put bounds on terms where $\gamma_1$ and $\gamma_2$ appear. 
The first question: how big are these sequences?
The equations in \eqref{AR1} which these satisfy  are examples of autoregressive models, specifically AR(1) processes \cite{grimmett-stirzaker}. We have the following almost sure estimate for solutions of such processes:
\begin{lemma} \label{AR estimate}
Suppose that $z(n)$, $n \ge 0$, are i.i.d.~random variables with zero mean and compact support. Fix $\theta \in (-1,1)$ and let
\be\label{ARgen}
\chi(n):=\sum_{k = 0}^{n-1} \theta^{k} z(n-k).
\ee
Then there exists a constant $C>0$ so that 
$$
\sup_{ n > 0}
{ |\chi(n)| \over \sqrt{ \ln(e+n) }} \le C \sqrt{ 1\over {1-\theta^2}}.
$$
The constant $C$ depends on the realization of $z(n)$ but does not depend on $\theta$; it is almost surely finite.
\end{lemma}

\begin{proof} 
The result is a consequence of of Hoeffding's inequality, whose proof can be found in \cite{hoeffding}:
\begin{theorem}\label{hoeffding}
Let $w(0),\dots,w(n-1)$ be mean-zero, independent random variables with $-b_k \le w(k) \le b_k$ almost surely and 
$
\chi(n) = \ds\sum_{k = 0}^{n-1} w(k).
$
Then for any $\mu \ge0$
$$
\P(|\chi(n)| \ge \mu) \le 2 \exp\left( -{ \mu^2 \over 2 \sum_{k=0}^{n-1} b_k^2}\right).
$$
\end{theorem}

We apply this to \eqref{ARgen}; let $w_n(k) := \theta^{k} z(n-k)$.
Since $\Ex[z(j)] = 0$ we have $\Ex[  w_n(k)]=0$ for all choices of $n$ and $k$. Since 
the $z(j)$ are independent it follows that, for fixed $n$, the  
 $w_n(k)$ 
are independent with respect to $k$. 
The support of $z(j)$ is compact so there is $a \ge 0$ for which the support lies in $[-a,a]$. 
Then the support of  $\theta^{k} z(n-k)$ is in $[-a\theta^k,a\theta^k]$. Thus $w_n(0),\dots,w_n(n-1)$ pass the hypotheses of Theorem \ref{hoeffding} with $b_k = a \theta^k$ and we have:
$$
\P[|\chi(n)| \ge \mu] \le 2 \exp\left( - { \mu^2 \over 2 a^2\sum_{k=0}^{n-1}  \theta^{2k}}\right)
=2 \exp\left( - { \mu^2 (1-\theta^2) \over 2 a^2 (1-\theta^{2n})}\right).
$$
Now let 
$\ds
\mu(n):= \sqrt{ \ln(e + n){4 a^2 (1-\theta^{2n}) \over 1-\theta^2 }}
$
so that
$$
\P[|\chi(n)| \ge \mu(n)] \le 2 \exp\left( - { \mu^2 (1-\theta^2) \over 2 a^2 (1-\theta^{2n})}\right)
=2 \exp\left( -2 \ln(e+n)\right) = {2 \over (e+n)^2}.
$$
Since $\sum_{n \ge 0} {2 / (e+n)^2}$ is finite, the Borel-Cantelli Lemma \cite{durrett} tells us that, almost surely, 
$|\chi(n)| \ge \mu(n)$ happens for at most finitely many $n$. For a given realization of $z(n)$ let $N_\om$ be the largest value of $n$ at which $|\chi(n)| \ge \mu(n)$ and put
$
c_\om:=\max_{1\le n \le N_\om} |\chi(n)|/\mu(n).
$
Thus we have
$$
|\chi(n)| \le  2a c_\om\sqrt{ \ln(e + n){1-\theta^{2n} \over 1-\theta^2 }}
\le 2a c_\om\sqrt{ {\ln(e + n)\over 1-\theta^2 }}
$$
for all $n$. Putting $C = 2a c_\om$ completes the proof.

\end{proof}

With Lemma \ref{AR estimate} we can prove
\begin{lemma}\label{gamma est} Take Hypothesis \ref{mass assumption} as given. Suppose that $\gamma_1(j)$ and $\gamma_2(j)$ solve \eqref{AR1} and $\gamma_1(0) = \gamma_2(0) = 0$. 
Then there exists a constant $C>0$ such that for all $\ep \in (0,1)$ we have
$$
\sup_{j \in \Z} {|\gamma_1(j)| +|\gamma_2(j)| \over \sqrt{\ln(e+|j|)} } \le C \ep^{-1/2}.
$$
The constant $C$ depends on the realization of the $\zeta(j)$ but is almost surely finite.
\end{lemma}
\begin{proof} 
We prove the estimate for $\gamma_2$ as the one for $\gamma_1$ is similar but easier.
Taking $j > 0$ in the second equation in \eqref{AR1} gives
$$
\gamma_2(j) -\gamma_2(j-1)=-\epsilon \gamma_2(j) + 
\left( \zeta(j) + S^-\zeta(j) + \zeta(j) \delta^+ \delta^- \zeta(j) + 2 \sigma^2\right)
$$
or rather
$$
\gamma_2(j) = {1 \over 1+ \ep} \gamma_2(j-1)
+{1 \over 1+ \ep} \left( \zeta(j) + S^-\zeta(j) + \zeta(j) \delta^+ \delta^- \zeta(j) + 2 \sigma^2\right).
$$
If we take $\gamma_2(0) =0$ then the we can find $\gamma_2(j)$ (for $j>0$) from the above by iteration. In particular we have
\bes
\begin{split}
\gamma_2(j) = &
{1 \over 1+\ep}
\sum_{k =0}^{j-1}
\theta_\ep^k \zeta(j-k)
+{1 \over 1+\ep}
\sum_{k =0}^{j-1}
\theta_\ep^k S^{-1}\zeta(j-k)\\
+&{1 \over 1+\ep}
\sum_{k =0}^{j-1}
\theta_\ep^k \left(\zeta(j-k) \delta^+ \delta^-\zeta(j-k) + 2 \sigma^2\right)\\
=:&{1 \over 1+\ep} \left(\gamma_{21}(j) + \gamma_{22}(j) + \gamma_{23}(j)\right)
\end{split}\ees
where we have put
$
\theta_\ep:= {1 /(1+ \ep)}.
$
To be clear $\gamma_{21}(j)$, $\gamma_{22}(j)$ and $\gamma_{23}(j)$ correspond to the three sums in the order of their appearance.

The random variables $\zeta(j)$ meet the hypotheses of Lemma \ref{AR estimate} and so we can apply the results to $\gamma_{21}(j)$ and $\gamma_{22}(j)$ forthwith to get:
$$
\sup_{j \in \Z} {|\gamma_{21}(j)| \over  \sqrt{\ln(e+|j|)}} \le C \sqrt{1 \over 1- \theta_\ep^2} 
\mand
\sup_{j \in \Z} {|\gamma_{22}(j)| \over  \sqrt{\ln(e+|j|)}} \le C \sqrt{1 \over 1- \theta_\ep^2} 
$$
for some $C>0$ which is almost surely finite.
An easy calculation shows that
$
{1 / (1- \theta_\ep^2) } = { (1 +\ep^2 )/ (2\ep + \ep^2)} < {1 / \ep}
$
when $\ep \in (0,1)$. Thus we have 
$$
\sup_{j \ge 0} {|\gamma_{21}(j)|+|\gamma_{22}(j)|  \over  \sqrt{\ln(e+|j|)}} \le C \ep^{-1/2}.
$$

Dealing with $\gamma_{23}(j)$ is a bit more complicated because the summands are not independent. 
We have $\gamma_{23}(j) = \sum_{k=0}^{j-1} \theta_\ep^k v(j-k)$ where
$$
v(j) = \zeta(j) \zeta(j+1) + \zeta(j) \zeta(j-1) - 2\zeta(j)^2 + 2 \sigma^2.
$$
From this we see that $v(j)$ and $v(j+1)$ are dependent. As are $v(j)$ and $v(j+2)$, since $\zeta(j+1)$ 
appears in both. But $v(j+3)$ and $v(j)$ have no terms in common and it follows that they are independent. Thus
$
\left\{ v(3 l) \right\}_{l \in \Z}
$
is an i.i.d.~collection of random variables. As are 
$
\left\{ v(3 l+1) \right\}_{l \in \Z}
$
and
$
\left\{ v(3 l+2) \right\}_{l \in \Z}.
$
We break up $\gamma_{23}(k)$ accordingly:
$$
\gamma_{23}(j) = 
\sum_{\substack{0\le k \le j-1\\ k=0  \textrm{mod} 3}}\theta_\ep^kv(j-k)
+\sum_{\substack{0\le k \le j-1\\ k=1  \textrm{mod} 3}} \theta_\ep^kv(j-k)
+\sum_{\substack{0\le k \le j-1\\ k=2 \textrm{mod} 3}}\theta_\ep^k v(j-k).
$$

Each of the three sums passes the hypotheses of Lemma \ref{AR estimate}, though there are some small subtleties. We estimate the first  as the others are all but the same.   Put $k = 3 l$ to find
\bes
\begin{split}
\sum_{\substack{0\le k \le j-1\\ k=0  \textrm{mod} 3}}\theta_\ep^kv(j-k)
& = 
\sum_{l= 0}^{\lfloor (j-1)/3 \rfloor} \theta_\ep^{3l} v(j-3l).
\end{split}
\ees
Then we have from Lemma \ref{AR estimate}:
$$
\sum_{l= 0}^{\lfloor (j-1)/3 \rfloor} \theta_\ep^{3l} v(j-3l) \le C \sqrt{\ln(e+\lfloor(j-1)/3\rfloor) }{\sqrt{1 \over 1-\theta_\ep^6}}.
$$
As $\theta_\ep^2>0$ we find $1/(1-\theta_\ep^6) = 1/(1-\theta_\ep^2)(1+\theta_\ep^2 + \theta_\ep^4) \le 1/(1-\theta_\ep^2) < 1/\ep$. This, along with the fact that $\ln$ is an increasing function, gives 
$$
\sum_{l= 0}^{\lfloor (j-1)/3 \rfloor} \theta_\ep^{3l} v(j-3l) \le C \sqrt{\ln(e+|j| )} \ep^{-1/2}
$$
which in turn leads to the estimate we are after.

We need estimates for $\gamma_2(j)$ when $j < 0$ too. If we take $j < 0$ in the second equation of \eqref{AR1} we  get
$$
\gamma_2(j) -\gamma_2(j-1)=\epsilon \gamma_2(j) + 
\left( \zeta(j) + S^-\zeta(j) + \zeta(j) \delta^+ \delta^- \zeta(j) + 2 \sigma^2\right).
$$
We rearrange this:
$$
\gamma_2(j-1) = (1-\ep) \gamma_2(j) -\left( \zeta(j) + S^-\zeta(j) + \zeta(j) \delta^+ \delta^- \zeta(j) + 2 \sigma^2\right).
$$
As we have taken $\gamma_2(0) = 0$ the above formula gives us $\gamma_2(-1)$ and, more generally, $\gamma_2(j)$, $j<0$, by iteration. For $j = -l < 0$ we obtain:
\bes\begin{split}
\gamma_2(-l) = -&\sum_{k= 0}^{l-1} \vartheta_\ep^k \zeta(-l+k+1)  -\sum_{k= 0}^{l-1} \vartheta_\ep^k S^{-1}\zeta(-l+k+1)\\
-&\sum_{k= 0}^{l-1} \vartheta_\ep^k\left(\zeta(-l+k+1) \delta^+ \delta^- \zeta(-l+k+1) + 2 \sigma^2\right)
\end{split}\ees
where $\vartheta_\ep := 1-\ep$. The first two sums pass the hypotheses of Lemma \ref{AR estimate} 
and since 
$
{1 /(1 - \vartheta_\ep^2)} = {1 / (2 \ep - \ep^2)} < {1/\ep} 
$
when $\ep \in (0,1)$ 
we can bound both as we did for $\gamma_{21}$ and $\gamma_{22}$ earlier. And the same skullduggery about independence that worked for $\gamma_{23}$ works for the third sum. All together we get
$$
\sup_{j \le 0} { |\gamma_{2}(j)| \over \sqrt{\ln(e+|j|)} }\le C \ep^{-1/2}.
$$
That completes the proof of Lemma \ref{gamma est}.
\end{proof}

Next we prove the main workhorse lemma for controlling $\gamma$ terms in our approximation:
\begin{lemma} \label{work it}
If $F = F(w,T)$ has $\sup_{|T|\le T_0}\| F(\cdot,T)\|_{H^1(1)}<\infty$ then
\be\label{gF est 1}
\sup_{|t| \le T_0/\ep^3} \| \gamma_k(\cdot) F(\ep(\cdot \pm t),\ep^3 t) \|_{\ell^2} \le C \ep^{-1} \sqrt{|\ln(\ep)|} \sup_{|T|\le T_0}\| F(\cdot,T)\|_{H^1(1)}
\ee
for $k =1,2$.

If in addition $\sup_{|T|\le T_0}\| F(\cdot,T)\|_{H^2(1)}<\infty$ then
\be\label{gF est 2}
\sup_{|t| \le T_0/\ep^3} \| E_0^\pm \left( \gamma_k(\cdot)  F(\ep(\cdot \pm t),\ep^3 t) \right)\|_{\ell^2} \le C \ep^{-1} \sqrt{|\ln(\ep)|} \sup_{|T|\le T_0}\| F(\cdot,T)\|_{H^2(1)}
\ee
for $k =1,2$. (The choices for $+$ or $-$ in $E_0^\pm$ and $F(\ep(\cdot \pm t),\ep^3 t)$ are not linked.)

The constant $C >0$ is almost surely finite.
\end{lemma}

\begin{proof}
First we tackle \eqref{gF est 1}.
We handle $k = 1$ and the ``$-$'' sign. The other cases are no different.
First
\bes\begin{split}
\| \gamma_1(\cdot)F(\ep(\cdot-t),\ep^3 t)\|_{\ell^2}^2 = \sum_{j \in \Z}
\gamma_1^2(j) F(\ep(j-t),\ep^3 t)^2.
\end{split}\ees
Using the estimate from Lemma \ref{gamma est} gets
\bes\begin{split}
\| \gamma_1(\cdot) F(\ep(\cdot-t),\ep^3 t)\|_{\ell^2}^2
\le& C \ep^{-1} \sum_{j \in \Z}
{\ln(e+|j|)} F(\ep(j-t),\ep^3 t)^2.
\end{split}\ees
A simple estimate leads us to
\bes\begin{split}
\| \gamma_1(\cdot) F(\ep(\cdot-t),\ep^3 t)\|_{\ell^2}^2.\le& C \ep^{-1} \sup_{j \in \Z } {\ln(e+|j|) \over1+(\ep(j-t))^2} \sum_{j \in \Z}
(1+(\ep(j-t))^2)F(\ep(j-t),\ep^3 t)^2.
\end{split}\ees
We can apply the second estimate in Lemma \ref{lwa} to the sum (with $u = (1+(\ep(j-t))^2)F(\ep(j-t),\ep^3 t)^2$) and find
\bes\begin{split}
\| \gamma_1(\cdot) F(\ep(\cdot-t),\ep^3 t)\|_{\ell^2}^2 
\le& C \ep^{-2} \sup_{j \in \Z }  {\ln(e+|j|) \over1+(\ep(j-t))^2} 
\| \sqrt{1+(\cdot)^2}
F(\cdot,\ep^3 t)\|_{H^1}^2\\
\le&C \ep^{-2} \sup_{j \in \Z }  {\ln(e+|j|) \over1+(\ep(j-t))^2} 
\| 
F(\cdot,\ep^3 t)\|_{H^1(1)}^2\\
\end{split}\ees
Thus
 $$
 \sup_{|t| \le T_0/\ep^3} \| \gamma_1(\cdot) F(\ep(\cdot-t),\ep^3 t)\|_{\ell^2}
 \le C \ep^{-1} \sqrt{  \sup_{|t| \le T_0/\ep^3}  \sup_{j \in \Z} {\ln(e+|j|) \over1+(\ep(j-t))^2} }
 \sup_{|T|\le T_0} \| 
F(\cdot,T)\|_{H^1(1)}.
 $$ 
 
From this we see that the proof of \eqref{gF est 1} will be complete once we show that there is
$C=C(T_0)>0$ such that $0 < \ep < 1$ implies
$$
{  \sup_{|t| \le T_0/\ep^3}  \sup_{j \in \Z} {\ln(e+|j|) \over1+(\ep(j-t))^2} } \le C|\ln(\ep)|.
$$
The proof is mainly elementary Calculus, but that does not mean it  is obvious.
Here are the details.
Let $f_\ep(y,t) :=\ds{ \ln(e+|y|) \over1+(\ep(y-t))^2}$. We show
$ \sup_{|t| \le T_0/\ep^3}  \sup_{y \in \R} f_\ep(y,t) \le C|\ln(\ep)|$. Since $f_\ep(y,t) =f_\ep(-y,-t)$  we have $\sup_{|t| \le T_0/\ep^3}  \sup_{y \in \R} f_\ep(y,t) =  \sup_{0\le t \le T_0/\ep^3}  \sup_{y \in \R} f_\ep(y,t)$. If $t \ge 0$ and $y \ge 0$ then $|y-t| \le |-y-t|$ which implies $f_\ep(-y,t) \le f_\ep(y,t)$. Thus
$\sup_{|t| \le T_0/\ep^3}  \sup_{y \in \R} f_\ep(y,t) =  \sup_{0\le t \le T_0/\ep^3}  \sup_{y \ge 0} f_\ep(y,t)$. 

Next we argue that $f_\ep(y,t)$ achieves its supremum at a point in $(0,\infty)$. Clearly $f_\ep(y,t)$ is non-negative and
$f_\ep(y,t) \to 0$ as $y \to \infty$. It is easy enough to show that $\lim_{y \to 0^+} \partial_y f_\ep(y,t) > 0$ when $t \ge 0$. Since $f_\ep$ is smooth (except at $y=0$), these considerations imply the existence of $y_\ep(t)\in(0,\infty)$ for which $f_\ep(y_\ep(t),t) =  \sup_{y \in \R} f_\ep(y,t)$ and $\partial_y f_\ep(y_\ep(t),t) = 0$.

So we search for solutions of $\partial_y f_\ep(y,t) = 0$ with $y \ge 0$. We claim that for $t \ge 0$ and $\ep \in (0,1)$
that
\be\label{claim}
\partial_y f_\ep(y,t)=0 \text{ and } y\ge 0 \implies  t \le y \le t+{1 \over \ep}.
\ee
Given the claim, $t \le y_\ep(t) \le t+\ep^{-1}$ follows and as such:
$$
f_\ep(y_\ep(t),t) = {\ln (e+y_\ep(t)) \over 1 + (\ep(y_\ep(t)-t))^2} \le 
\ln(e + t + \ep^{-1}).
$$
In turn we have $\sup_{0\le t \le T_0/\ep^3}  \sup_{y \ge 0} f_\ep(y,t) \le \ln(e+T_0 \ep^{-3} + \ep^{-1}) \le C |\ln(\ep)|$
for a constant depending only on $T_0$. 

So we will be done if we establish the claim. Routine computations show that $\partial_y f_\ep(y,t) = 0$ if and only if 
\be\label{turd}
{1 + \ep^2 (y-t)^2 \over 2 \ep^2 (y-t)} = (e+y) \ln(e+y).
\ee
Note that if $ 0 \le  y < t$ then the left-hand side of \eqref{turd} is negative whereas the right-hand side is positive. So there can be no solutions with $y < t$ and this implies the left-hand inequality in \eqref{claim}.

For the right-hand inequality, let us assume that $y -t > \ep^{-1}$ and $t \ge 0$. This gives
$$
{1 + \ep^2 (y-t)^2 \over 2 \ep^2 (y-t)} = {1 \over 2 \ep^2 (y-t)} + {y -t \over 2} < {1 \over 2 \ep} + {y -t \over 2}.
$$
Next, since $t \ge 0$ then $y-t \le y$ which implies 
$\ds
{1 + \ep^2 (y-t)^2 \over 2 \ep^2 (y-t)}  < {1 \over 2 \ep} + {y  \over 2}.
$
Since $y - t > \ep^{-1}$ and $t \ge 0$ we have $y>\ep^{-1}  $. Thus
$\ds
{1 + \ep^2 (y-t)^2 \over 2 \ep^2 (y-t)}  < y.
$
Also we clearly have $y < (e+y) \ln(e+y)$ and so all told $$
{1 + \ep^2 (y-t)^2 \over 2 \ep^2 (y-t)}  <  (e+y) \ln(e+y)
$$
when $y - t > \ep^{-1}$ and $t \ge 0$. This precludes $\partial_y f_\ep(y,t) =0$ and
 the right inequality in \eqref{claim} follows. Thuse we are done with the proof of \eqref{gF est 1}.

The estimate \eqref{gF est 2} follows from \eqref{gF est 1} with a few tricks. First we have by direct calculation
$
 E_0^+ \left( \gamma_1  F\right)= \gamma_1(j+1) (E_0^+ F). 
$
Second, if we let $\I_\ep G(X)= \ep^{-1} \int_X^{X+\ep} G(Y) dY$ then the Fundamental Theorem of Calculus and the definition of $E_0^+$ tell us $E_0^+ F = \I_\ep \partial_w F$. One can show (see the argument that leads to equation (3.4) in \cite{mcginnis-wright}) that $\|\I_\ep G\|_{H^n(r)} \le C\|G\|_{H^n(r)}$. Putting it all together we get
\bes\begin{split}
\sup_{|t| \le T_0/\ep^3} \| E_0^\pm \left( \gamma_k(\cdot)  F(\ep(\cdot \pm t),\ep^3 t) \right)\|_{\ell^2} &=
\sup_{|t| \le T_0/\ep^3} \| \gamma_k(\cdot+1)  \I_\ep \partial_w F(\ep(\cdot \pm t),\ep^3 t) \|_{\ell^2}\\
&\le C \ep^{-1}\sqrt{|\ln(\ep)|}\sup_{|T|\le T_0}\|  \I_\ep \partial_w F(\cdot,T)\|_{H^1(1)}\\
&\le C \ep^{-1} \sqrt{|\ln(\ep)|}\sup_{|T|\le T_0}\|   F(\cdot,T)\|_{H^2(1)}.
\end{split}
\ees
 \end{proof}
 
 Now we can control all the $\gamma$ dependent terms.
 
\begin{lemma}\label{gamma prop} Assume Hypothesis \ref{mass assumption}.
Let $\tilde{q}_\ep(j,t)$ and $\tilde{p}_\ep(j,t)$ be the {extended KdV approximators}
as in Definition \ref{full approx} where we further assume that that $A$ and $B$ are good solutions of KdV on $[-T_0,T_0]$. Then almost surely $$
\sup_{|t| \le T_0 \ep^{-3}} \left(
\|P_3(\cdot,\ep \cdot, \ep t, \ep^3 t)\|_{\ell^2} +
\|Q_{3\gamma}(\cdot,\ep \cdot, \ep t, \ep^3 t)\|_{\ell^2}\right)= \O(\ep^{-1}\sqrt{|\ln(\ep)|})
$$
and
\bes\begin{split}\sup_{|t| \le T_0 \ep^{-3}} \left(
\|W_{1\gamma}(\cdot,\ep \cdot, \ep t, \ep^3 t)\|_{\ell^2}+
\|W_{2\gamma}(\cdot,\ep \cdot, \ep t, \ep^3 t)\|_{\ell^2}\right)=  \O(\ep^{-1}\sqrt{|\ln(\ep)|}).
\end{split}\ees

\end{lemma}

\begin{proof}
The estimates for $P_3$ and $Q_{3\gamma}$ are immediate from Lemma \ref{work it}
and their definitions.
A direct calculation shows that 
$$
W_{1\gamma} = E_0^+P_3 - \partial_\tau Q_{3 \gamma} - \ep^2 \partial_T Q_{3 \gamma}.
$$
Each of these can be estimated with Lemma \ref{work it} as well.
Another calculation gives
\bes\begin{split}
W_{2\gamma} = 
&E^-_0 Q_{3 \gamma} - m \partial_\tau P_3 - m \ep^2\partial_T P_3\\
&+ 2\ep^2 D^- (Q_{3\gamma}Q_1)+2\ep^3D^- (Q_{3\gamma} Q_2) + 2 \ep^4 D^- (Q_{3 \gamma} Q_{30}) + \ep^4 D^- Q_{3\gamma}^2.
\end{split}\ees
The first line of the above we estimate with Lemma \ref{work it}. The ones in the second line all hinge on estimating terms of the form
$D^-(Q_{3 \gamma} Q_l)$ for different choices of $l$. The definition of $D^-$ and the triangle inequality give
$
\|D^- (Q_{3 \gamma} Q_l)(\cdot,\ep \cdot )\|_{\ell^2}  \le2 \|Q_{3\gamma}(\cdot,\ep \cdot )  Q_l(\cdot,\ep \cdot )\|_{\ell^2}
$
Then we use the fact that $\|f g\|_{\ell^2} \le \|f\|_{\ell^2} \|g\|_{\ell^2}$ to get
$
\|D^- (Q_{3 \gamma} Q_l)(\cdot,\ep \cdot )\|_{\ell^2} \le2 \|Q_{3\gamma}(\cdot,\ep \cdot )\|_{\ell^2}\|  Q_l(\cdot,\ep \cdot )\|_{\ell^2}.
$
At this point the remainder of the estimates follow from earlier estimates on the component $Q_k$ and bookkeeping.

\end{proof}

\subsection{Finishing up}
We are now in position to prove Proposition \ref{big prop}. 
\begin{proof}
We begin with $\alpha_1(\ep)$. From its definition, \eqref{start}, \eqref{Qgamma} and the triangle inequality we have
\bes \begin{split}
\alpha_1(\ep) \le &\sum_{n=0}^2 \ep^{n+2} \sup_{|t| \le T_0/\ep^3}  \left(\|Q_n(\cdot,\ep \cdot, \ep t, \ep^3 t)\|_{\ell^2}
+\|P_n(\cdot,\ep \cdot, \ep t, \ep^3 t)\|_{\ell^2}\right)\\
+& \ep^5 \sup_{|t| \le T_0/\ep^3} \|Q_{30}(\cdot,\ep \cdot, \ep t, \ep^3 t)\|_{\ell^2}
+\ep^5 \sup_{|t| \le T_0/\ep^3}\left(\|Q_{3\gamma}(\cdot,\ep \cdot, \ep t, \ep^3 t)\|_{\ell^2}+ \|P_3(\cdot,\ep \cdot, \ep t, \ep^3 t)\|_{\ell^2}\right)
\end{split} \ees
Then Lemmas \ref{res prop1} and \ref{gamma prop} gives us
 $
\alpha_1(\ep)  = \O(\ep^{3/2}). 
$

For $\alpha_3(\ep)$, we use its definition, Lemma \ref{residual calculation}, \eqref{wt} and the triangle inequality to obtain
\bes\begin{split}
\alpha_3(\ep) \le &\ep^6 \sup_{|t| \le T_0/\ep^3}
\left( \|\gamma_1(\cdot) \partial_X^3 P_0(\cdot,\ep \cdot,\ep t,\ep^3 t) \|_{\ell^2}
+\|\gamma_2(\cdot) \partial_X^3 Q_0(\cdot,\ep \cdot,\ep t,\ep^3 t) \|_{\ell^2}\right)\\
+ &\ep^6 \sup_{|t| \le T_0/\ep^3}
\left( \|W_{10}(\cdot,\ep \cdot,\ep t,\ep^3 t) \|_{\ell^2}
+\|W_{20}(\cdot,\ep \cdot,\ep t,\ep^3 t) \|_{\ell^2}\right)\\
+&\ep^6 \sup_{|t| \le T_0/\ep^3}
\left( \|W_{1\gamma}(\cdot,\ep \cdot,\ep t,\ep^3 t) \|_{\ell^2}
+\|W_{2\gamma}(\cdot,\ep \cdot,\ep t,\ep^3 t) \|_{\ell^2}
\right).
\end{split}\ees
Then  Lemmas \ref{res prop1} , \ref{work it} and \ref{gamma prop} give
 $
\alpha_3(\ep)  = \O(\ep^5\sqrt{|\ln(\ep)|}). 
$

To prove the estimate for  $\alpha_2(\ep)$, from \eqref{start} we have
$$
\partial_t \tilde{q}_\ep = \ep^3 \partial_\tau Q_0 + 
\bunderbrace{\ep^5 \partial_T Q_0 + \ep^4 \sum_{k=1}^3 \left(\ep^{k} \partial_\tau Q_k +\ep^{k+2} \partial_T Q_k\right)}{\ep^4 \tilde{h}_\ep}.
$$
Since $\|f\|_{\ell^\infty} \le \|f\|_{\ell^2}$ we have
$
\|\tilde{h}_\ep\|_{\ell^\infty} \le \| \tilde{h}_\ep\|_{\ell^2}.
$
All terms appearing in $\tilde{h}_\ep$ have been estimated in one place or another previously and each is $\O_{\ell^{2}}(\ep^{-1/2})$ at worst so that we get $\sup_{|t| \le T_0/\ep^3} \| \ep^4 \tilde{h}_\ep\|_{\ell^\infty} = \O(\ep^{7/2})$. On the other hand using the first estimate in Lemma \ref{lwa} shows that $$
\sup_{|t|\le T_0/\ep^3} \|\ep^3 \partial_\tau Q_0(\cdot,\ep \cdot, \ep t,\ep^3 t)\|_{\ell^\infty} \le 
\sup_{|t|\le T_0/\ep^3} \ep^3\left( \|A_w(\cdot,\ep^3t)\|_{W^{1,\infty}}
+ \|B_l(\cdot,\ep^3t)\|_{W^{1,\infty}}\right) \le C\ep^3.
$$
So all told we have $\alpha_2(\ep) = \O(\ep^3)$.

Next, if we let \be\label{checks}
\check{q}_\ep(j,t):=\sum_{k=1}^3 \ep^{k+2} Q_k(j,\ep j, \ep t,\ep^3 t)
\mand
\check{p}_\ep(,t):=\sum_{k=1}^3 \ep^{k+2} P_k(j,\ep j, \ep t,\ep^3 t)
\ee
then the estimates from Lemmas \ref{res prop1} and \ref{gamma prop} lead to
\be\label{extras}
\sup_{|t|\le T_0/\ep^3} \|\check{q}_\ep,\check{p}_\ep\|_{\ell^2} = \O(\ep^{5/2}).
\ee
So long $A$ and $B$ are not both identically zero it is easy using the conservation laws of KdV to find that
$\inf_{|T| \le T_0} \|A(\cdot,T)\|_{H^7(1)} + \|B(\cdot,T)\|_{H^7(1)} \ge b > 0$, for some $b$. This leads, by the triangle inequality, to
$
\beta_1(\ep)=\inf_{|t| \le T_0/\ep^3} \|\tilde{q}_\ep,\tilde{p}_\ep\|_{\ell^2} \ge C \ep^{3/2}.
$
This completes the proof.
\end{proof}

\section{The main event}\label{main event}

Now we can state and prove our main theorem in full detail.
\begin{theorem}\label{main theorem}
Let $m(j)$ be a realization of the mass coefficients subject to Hypothesis \ref{mass assumption}.
Fix $T_0>0$ and $\Phi,\Psi \in H^7(1)$. Let $(q(j,t),p(j,t))$ be the solution of 
the transparent random mass FPUT lattice \eqref{FPUT}
with initial data
$$
q(j,0) = \ep^2 \Phi(\ep j) \mand p(j,0) = \ep^2 \Psi(\ep j).
$$
Let $A(w,T)$ and $B(l,t)$ be the solutions of the KdV equations \eqref{KdVA} and \eqref{KdVB} with initial data
$$
A(w,0) = {1 \over 2}\Phi(w) -{1 \over 2}\Psi(w) \mand B(l,0) = {1 \over 2}\Phi(l) +{1 \over 2}\Psi(l).
$$
Then there exits $\ep_\star = \ep_\star(m,T_0,\Phi,\Psi)$ (almost surely positive) and $C_\star = C_\star(m,T_0,\Phi,\Psi)>0$ (almost surely finite) such that, for all $\ep \in (0,\ep_\star)$, we have the absolute $\ell^2$-error estimates
\bes\begin{split}\label{abs err est}
&\sup_{|t| \le T_0 /\ep^3}{ \left\|q(\cdot,t) - \ep^2\left[ 
A(\ep(\cdot-t),\ep^3 t) + B(\ep(\cdot+t),\ep^3 t)
\right]\right\|_{\ell^2}}\le C_\star \ep^2 \sqrt{|\ln(\ep)|} \mand \\
&\sup_{|t| \le T_0 /\ep^3}{\left\|p(\cdot,t) - \ep^2\left[ 
-A(\ep(\cdot-t),\ep^3 t) + B(\ep(\cdot+t),\ep^3 t)
\right]\right\|_{\ell^2} }\le C_\star \ep^2 \sqrt{ |\ln(\ep)|}.\end{split}
\ees
If at least one of $\Phi$ or $\Psi$ is non-zero then the associated relative $\ell^2$-error estimates are $\O(\sqrt{\ep|\ln(\ep)|})$.

\end{theorem}

\begin{proof}
Take $A$ and $B$ with the initial data as in the statement and form $\tilde{q}_\ep$ and $\tilde{p}_\ep$ 
as in Definition \ref{full approx}. Then we have the estimates for $\alpha_1(\ep)$, $\alpha_2(\ep)$, $\alpha_3(\ep)$ and $\beta_1(\ep)$ as in Proposition \ref{big prop}, which is to say we have met condition \eqref{alf} from the statement of Theorem \ref{approx thm}.

Note that 
$$
q(j,t)-\tilde{q}_\ep(j,t) =q(j,t)- \ep^2 [A(\ep(j-t),\ep^3 t) + B(\ep(j+t),\ep^3 t)] - \check{q}_\ep(j,t)
$$
where $\check{q}_\ep$ is given in \eqref{checks}. In \eqref{extras} we showed that $\check{q}_\ep = \O_{\ell^2}(\ep^{5/2})$ for $|t| \le T_0/\ep^3$. The initial conditions for $p$, $q$, $A$ and $B$ are arranged so that
$$
q(j,0)-\tilde{q}_\ep(j,0) =  -\check{q}_\ep(j,0)
$$
And so we have $\|q(0)-\tilde{q}_\ep(0)\|_{\ell^2} = \|\check{q}_\ep(0)\|_{\ell^2} =\O(\ep^{5/2})$. Similarly we have
$$
p(j,t)-\tilde{p}_\ep(j,t) =p(j,t)- \ep^2 [-A(\ep(j-t),\ep^3 t) + B(\ep(j+t),\ep^3 t)] - \check{p}_\ep(j,t)
$$
and $\|p(0)-\tilde{p}_\ep(0)\|_{\ell^2} = \|\check{p}_\ep(0)\|_{\ell^2} =\O(\ep^{5/2})$. 

Since $\alpha_3(\ep)/\ep^3 = \O(\ep^2 \sqrt{|\ln(\ep)|})$ and $\ep^{5/2} = o(\ep^2 \sqrt{|\ln(\ep)|})$, these estimates imply that we meet the condition on the initial data in the statement of Theorem \ref{approx thm}. Thus we have the conclusion
$$
\sup_{|t|\le T_0/\ep^3} \|q(t)-\tilde{q}_\ep(t),p(t)-\tilde{p}_\ep(t)\|_{\ell^2} = \O\left(\ep^2 \sqrt{|\ln(\ep)|}\right).
$$
Then the triangle inequality plus \eqref{checks} give
\bes
\sup_{|t| \le T_0 /\ep^3}{ \left\|q(\cdot,t) - \ep^2\left[ 
A(\ep(\cdot-t),\ep^3 t) + B(\ep(\cdot+t),\ep^3 t)
\right]\right\|_{\ell^2}}  = \O(\ep^2 \sqrt{|\ln(\ep)|}) 
\ees
and
\bes
\sup_{|t| \le T_0 /\ep^3}{\left\|p(\cdot,t) - \ep^2\left[ 
-A(\ep(\cdot-t),\ep^3 t) + B(\ep(\cdot+t),\ep^3 t)
\right]\right\|_{\ell^2}  }= \O(\ep^2 \sqrt{|\ln(\ep)|}) .
\ees
This is the absolute error estimate in the theorem. The relative error estimate follows from the estimate on $\beta_1(\ep)$.
\end{proof}

\section{Numerics}\label{numerics}

In this section we report the outcomes of a variety of numerical simulations of solutions of \eqref{FPUT}. In all cases our methodology is to truncate \eqref{FPUT} to $|j|\le M$ where $M \gg 1$ and enforce periodic boundary conditions ($M$ is always taken to be so incredibly vast that the solutions are never large anywhere near the edges of the computational domain). The resulting system is a large finite-dimensional ODE which we solve with a standard RK4 algorithm. This is essentially the same method as used in \cite{GMWZ,mcginnis-wright}. The calculations were performed in MATLAB.

\subsection{Amplitude attenuation}
The first experiment simulates \eqref{FPUT} with a number of choices for $m(j)$. These are:
\begin{itemize}
\item $m(j) =1$ for all $j$, that is, they are constant.
\item $m(j) = 1 + (-1)^j/4$, {\it i.e.} $2-$periodic.
\item $m(j)$ meet the transparency condition \eqref{transparency} where $\zeta(j)$ are drawn from the uniform distribution on $[-1/8,1/8]$.
\item $m(j)$ are i.i.d.~random variables, drawn uniformly from $[1/2,3/2]$.
\end{itemize}
For all these cases, we choose as initial conditions:
\be\label{ic1}
q(j,0) = 3\ep^2 \sech^2(\sqrt{6} \ep j) \mand p(j,0) = -3\ep^2 \sech^2(\sqrt{6} \ep j).
\ee
We  take $\ep =1/2, 1/4, 1/8$ and $1/16$ and simulate from $t=0$ out to $t=3/\ep^3$.

Famously, solutions of KdV equations with smooth and localized initial data will, over time, resolve into the sum of separated solitary waves of fixed amplitude \cite{drazin-johnson}. Thus, if the solution of the FPUT lattice is well-approximated by a KdV equation we 
expect the $\ell^\infty$-norm to at least roughly stabilize over long time periods. And so in Figure \ref{avt} we plot $\|q(\cdot,T/\ep^3),p(\cdot,T/\ep^3)\|_{\ell^\infty}/\ep^2$ vs $T$, for $0\le T\le3$. (The scaling here is
to be consistent with the long wave scaling so that we may compare various choices of $\ep$ on the same plot.) We see exactly this stabilization in the plots for the constant, $2$-periodic and transparent cases. Furthermore,  the stabilization becomes more pronounced as $\ep$ decreases, which is consistent with the rigorous KdV approximation theorems here and in \cite{schneider-wayne,bruckner-etal,GMWZ}. But when the masses are taken to be i.i.d., there is an obvious, pronounced decay of the amplitude; this attenuation (up to the scaling) becomes stronger as $\ep$ decreases. This is why we said in the Introduction that a KdV approximation for the i.i.d.~problem is not appropriate. 

\begin{figure}
	\centering
	\begin{subfigure}{.45\textwidth}
		\includegraphics[width=\textwidth]{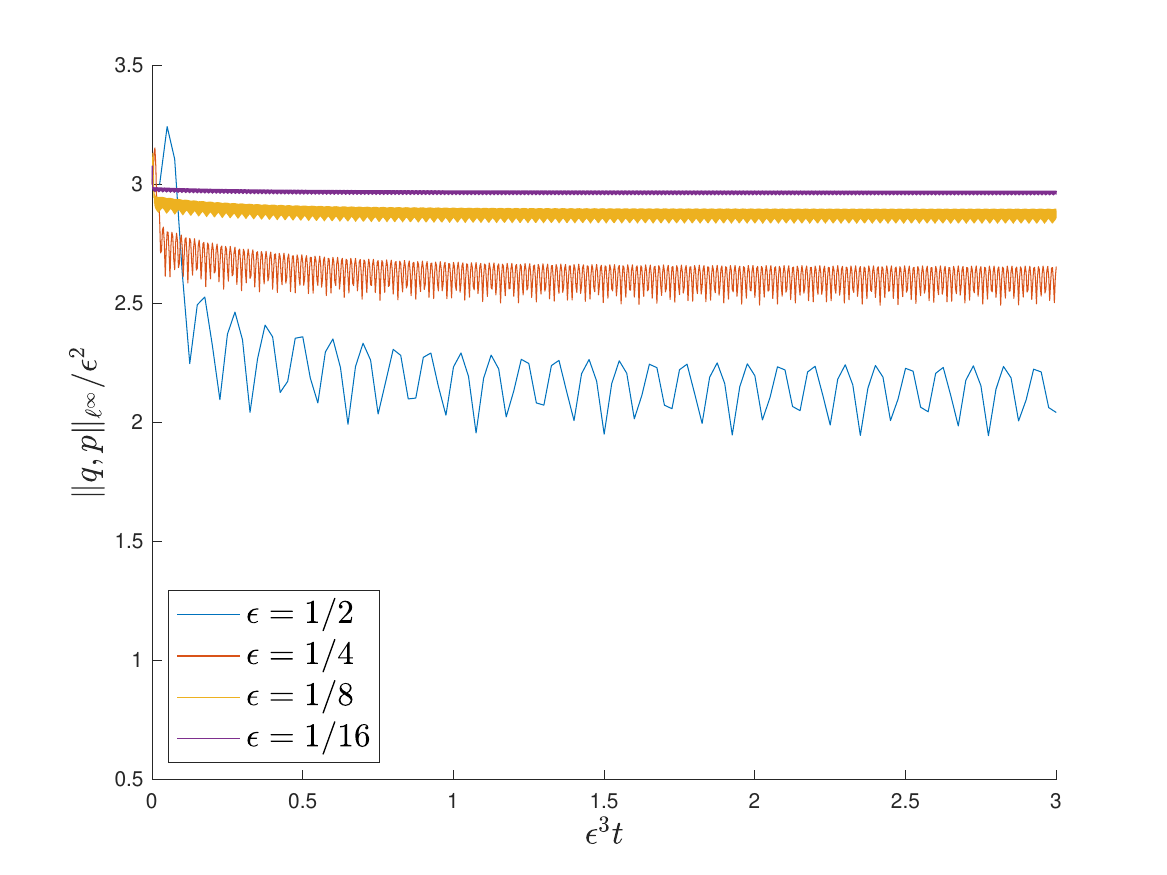}
		\caption{$m(j)$ are constant}
	\end{subfigure}
	\begin{subfigure}{.45\textwidth}
		\includegraphics[width=\textwidth]{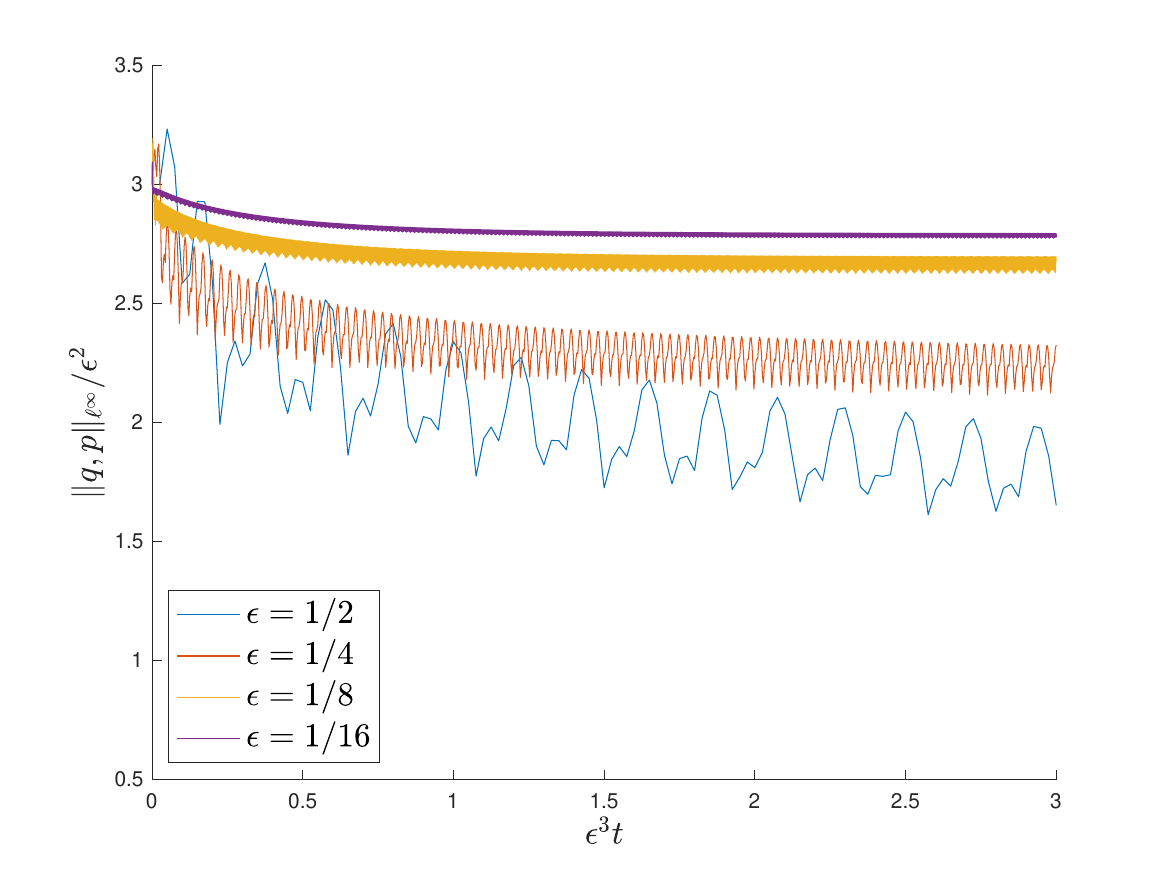}
		\caption{$m(j)$ are $2$-periodic}
	\end{subfigure}

	\begin{subfigure}{.45\textwidth}
		\includegraphics[width=\textwidth]{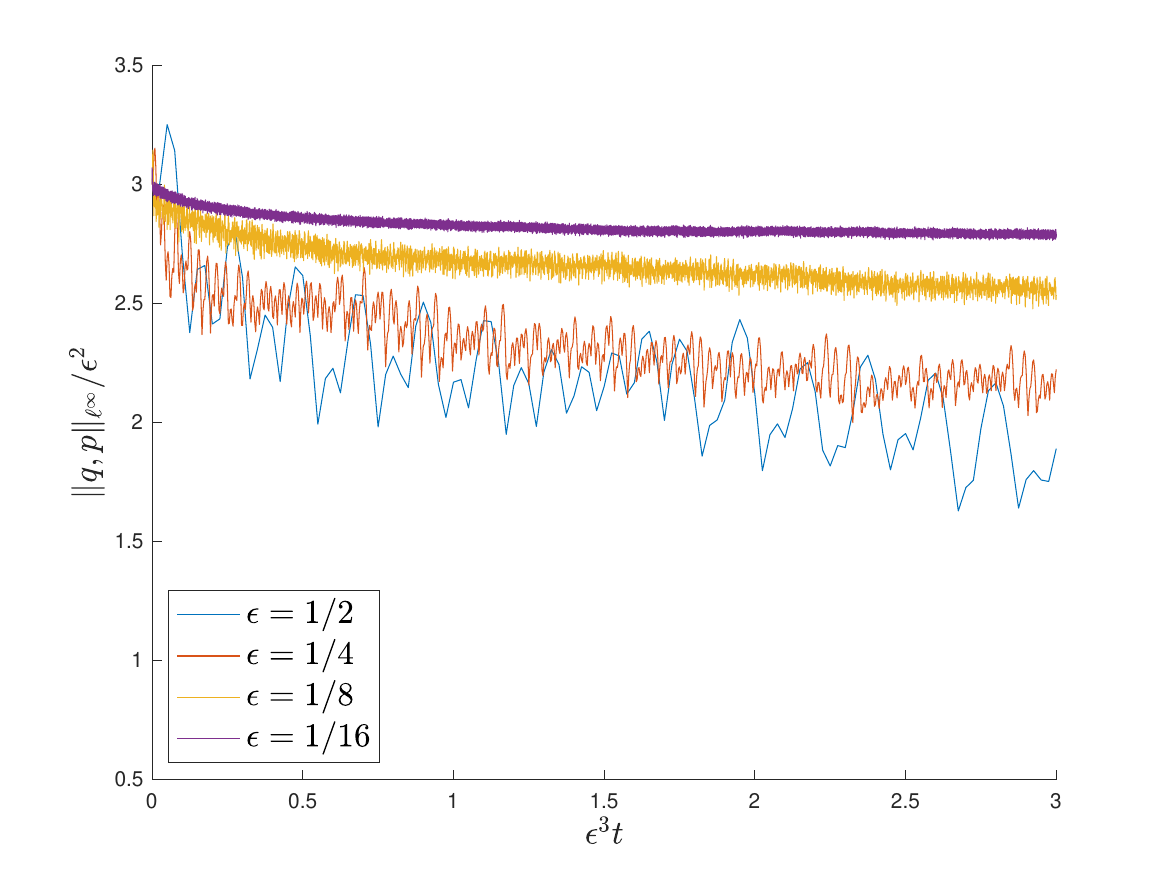}
		\caption{$m(j)$ are transparent}
	\end{subfigure}
	\begin{subfigure}{.45\textwidth}
		\includegraphics[width=\textwidth]{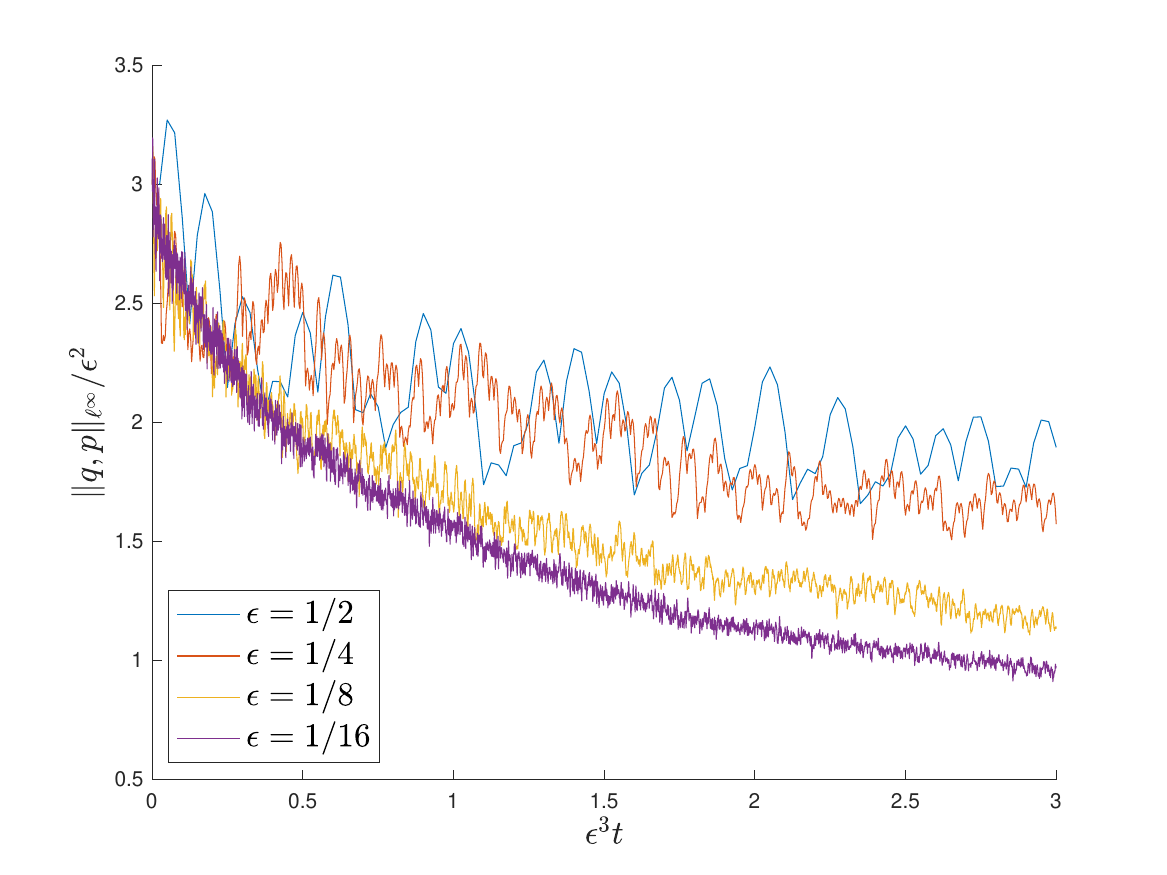}
		\caption{$m(j)$ are i.i.d.}
	\end{subfigure}
	
	\caption{Scaled $\ell^\infty$-amplitude vs scaled time for solutions of \eqref{FPUT} with long-wave data as in \eqref{ic1}.}\label{avt}
\end{figure}

\subsection{Numerical computation of optimal error bound}

In the second experiment we aim to corroborate the conclusions of our main result, Theorem \ref{main theorem}. We  simulate \eqref{FPUT} with $m(j)$ subject to the transparency condition \eqref{transparency} where $\zeta(j)$ are drawn from the uniform distribution on $[-1/8,1/8]$. In this case, $\sigma^2 = 1/192$.
We choose the initial data so that $B(l,T)$ is zero and $A(w,T)$ is an exact solitary wave solution of \eqref{KdVA}, namely
$3 \sech^2 \left({\sqrt{6 \over 1+ 24 \sigma^2} (w-T) }\right)$.
That is to say, we take
\be\label{ic2}
q(j,0) =3\ep^2 \sech^2\left({\sqrt{6 \over 1+ 24 \sigma^2} \ep j  }\right) \mand p(j,0) =-3\ep^2 \sech^2\left({\sqrt{6 \over 1+ 24 \sigma^2} \ep j  }\right).
\ee
\begin{figure}
	\centering
	\begin{subfigure}{.65\textwidth}
		\includegraphics[width=\textwidth]{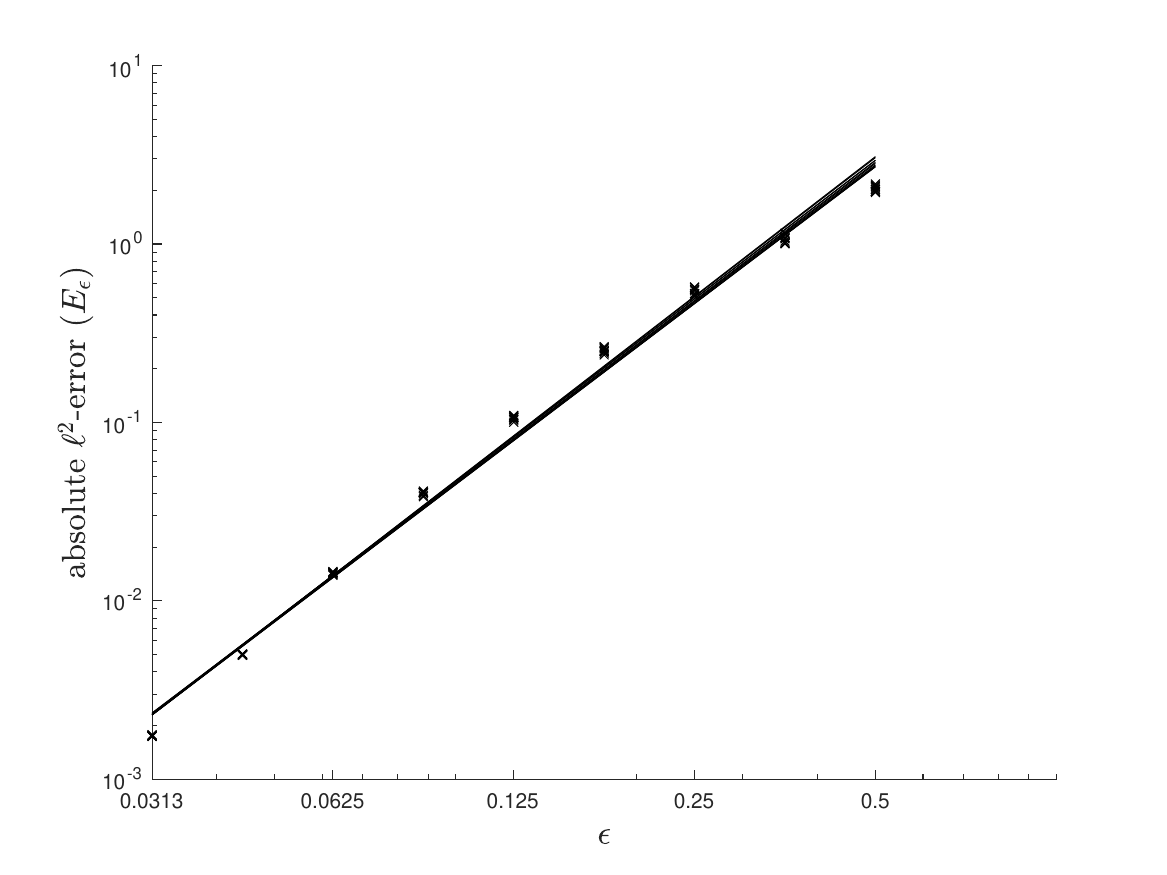}
	\end{subfigure}
	\caption{Ten realizations of absolute $\ell^2$-error vs $\ep$, loglog plot, for solutions of \eqref{FPUT} with ``solitary wave-like'' data as in \eqref{ic2}.}\label{eerrvsep}
\end{figure}
We simulate for $\ep = 2^{-l/2}$ where $l=2,\dots,10$ and run the simulations from $t = 0$ to $t = 3/\ep^3$. 
(When $\ep=1/32$ this takes a very long time!)
To be clear, we fix a realization and then vary $\ep$ as stated with the same realization used throughout.
Then we compute the overall absolute error, specifically:
\bes\begin{split}
E_\ep:=&\sup_{0 \le t \le T_0/\ep^3}
\left
\| q(\cdot,t) - 3 \ep^2 \sech^2 \left({\sqrt{6 \over 1+ 24 \sigma^2} \left(\ep(\cdot-t-\ep^2t)\right) }\right)
\right\|_{\ell^2}\\+
&\sup_{0 \le t \le T_0/\ep^3}
\left
\| p(\cdot,t) + 3 \ep^2 \sech^2 \left({\sqrt{6 \over 1+ 24 \sigma^2} \left(\ep(\cdot-t-\ep^2t)\right) }\right)
\right\|_{\ell^2}.
\end{split}\ees
Then we repeat for another realization (ten different realizations all together).
If we plot $E_\ep$ vs $\ep$ on a loglog plot, Theorem \ref{eerrvsep} tells us the best fit line to the data should have slope somewhere around $2$, or larger. The results are shown 
 in Figure \ref{eerrvsep}; all ten realizations on the same graph. The line of best fit has slope exceeding $2.5$ in each case; they are $2.5715$, $2.5982$,    $2.5677$,  $2.5445$, $2.5778$, $2.5948$, $2.5510$, $2.5760$, $2.5499$, $2.5563$.
  
This numerically computed slope is over $.5$ larger than what we expect from our rigorous estimate. That is to say, the numerics indicate that the approximation of the transparent mass FPUT lattice by KdV is a fair bit better than what our results from Theorem \ref{main theorem} show. We repeated the same experiment for many different realizations and the numerically computed slope was always near to $2.5$. Thus we conjecture that the absolute $\ell^2$-error is at worst $\O(\ep^{5/2})$, which is the same size as the error for the constant and periodic problems. Of course we do not know how to prove such a thing at this time.

\section{What's next.}\label{next}
Our results are the first piece of much larger program aimed at bringing stochastic homogenization to nonlinear dispersive problems. Here are a number of open problems, some of which should be relatively straightforward given our results here and others of which will require substantial new technical ideas.
\begin{enumerate}
\item Prove a result analogous to Theorem \ref{main theorem} but in expectation instead of in the almost sure sense. In our work on the linear i.i.d.~lattice \cite{mcginnis-wright} we proved approximation results in both senses and the strategy for the expectation result is almost surely transferable.
\item Study \eqref{FPUT} but allow spatial heterogeneity in the spring potentials as well. We expect an analogous transparency condition can be used to achieve a result similar to the one here.
\item Confirm (or reject!) the conjecture that the sharp order of the KdV approximation error for the transparent mass FPUT lattice is smaller than $\O(\ep^{5/2}$). The error estimate we prove here is due entirely to our use of the autoregressive processes in the extended approximation. But perhaps a yet more clever option exists to handle the terms which we encountered at $Z_{15}$ and $Z_{25}$.
\item If one can get the conjectured sharp error estimate, it opens the door to replacing the transparency condition \eqref{transparency} with  $$m(j) = 1+\delta^- \zeta(j).$$ 
Here there is only one finite-difference on $\zeta(j)$ (which are still i.i.d.) instead of two. Some preliminary simulations indicate this condition, which we call the {\it translucent random mass FPUT lattice}, should have a valid KdV approximation. 
\item The transparency condition is hardly natural or obvious.
This raises the question: are there conditions that one can place on $m(j)$ that are less rigid than the transparency (or translucency) conditions? Note that the transparency condition implies that the $m(j)$ are correlated with one another; perhaps strictures placed upon correlation lengths in $m(j)$ can be used to get results similar to our results here.
\item How can one rigorously capture the amplitude attenuation seen in the numerics for the i.i.d.~random mass problem in Figure \ref{avt}?
Clearly KdV is the wrong approach, but perhaps there is some other modulation equation that can capture the dynamics. The articles \cite{flach-etal-1,flach-etal-2} suggest the use of nonlinear diffusion equations. On the other hand, an analysis of a spectral problem associated with linear random lattices by \cite{besing} 
indicates a connection to Anderson localization and the authors of \cite{lukkarinen-spohn} utilize Boltzmann models in a high-dimensional linear version of \eqref{FPUT}.
\item Can our approach be carried over to the problem of long water waves over random bathymetry? Our transparency condition is inspired by a similar condition on the bathymetry proposed in \cite{papa-rosales}. Can a continuous version of an autoregressive process (that is, an {\it Ornstein-Uhlenbeck process}) be used to control the residuals in that problem and prove a rigorous approximation? 
\item How about nonlinear wave equations with random coefficients? Nonlinear Schr\"odinger equations (discrete or continuous)? Any of the multitude of equations named for Joseph Valentin Boussinesq? Or problems in higher spatial dimensions? 
\end{enumerate}

\bibliographystyle{plain}
\bibliography{transparent-fput}{}
\end{document}